\documentclass[a4paper,11pt,reqno]{amsart}

\usepackage{amssymb}
\usepackage{latexsym}
\usepackage{amsmath}
\usepackage{mathrsfs}
\usepackage{euscript}
\usepackage{amsthm}
\usepackage{upgreek}
\usepackage{cite}
\usepackage{xcolor}
\usepackage{tikz}
\usepackage{calc}                         

\newcommand{\scal}[2]{\langle #1,#2\rangle}

\newcommand{\rr}[1]{\mathbf R^{#1}}
\newcommand{\zz}[1]{\mathbf Z^{#1}}

\newcommand{\nn}[1]{\mathbf N^{#1}}

\newcommand{\nm}[2]{\Vert #1\Vert _{#2}}
\newcommand{\NM}[2]{\left \Vert #1\right \Vert _{#2}}

\newcommand{\sets}[2]{\{ \, #1\, ;\, #2\, \} }

\newcommand{\ep}{\varepsilon}
\newcommand{\fy}{\varphi}

\newcommand{\cdo}{\, \cdot \, }

\newcommand{\supp}{\operatorname{supp}}

\newcommand{\eabs}[1]{\langle #1\rangle}

\newcommand{\vrum}{\vspace{0.1cm}}

\newcommand{\sinc}{\operatorname{sinc}}

\newcommand{\sfW}{\mathsf{W}}

\newcommand{\maclD}{\mathcal D}

\newcommand{\maclS}{\mathcal S}

\newcommand{\mascF}{\mathscr F}

\newcommand{\mascP}{\mathscr P}
\newcommand{\mascS}{\mathscr S}

\newcommand{\mabfj}{\boldsymbol j}
\newcommand{\mabfk}{\boldsymbol k}

\newcommand{\fkb}{\mathfrak b}

\numberwithin{equation}{section}          
\newtheorem{thm}{Theorem}
\numberwithin{thm}{section}

\newcommand{\rubrik}{}
\newtheorem{prop}[thm]{Proposition}
\newtheorem{cor}[thm]{Corollary}
\newtheorem{lemma}[thm]{Lemma}

\theoremstyle{definition}

\newtheorem{defn}[thm]{Definition}

\theoremstyle{remark}

\newtheorem{rem}[thm]{Remark}              

\author{Joachim Toft}

\address{Department of Mathematics,
Linn{\ae}us University, V{\"a}xj{\"o}, Sweden}

\email{joachim.toft@lnu.se}

\title[Step multipliers on modulation spaces]
{Step multipliers, Fourier step multipliers and multiplications on
quasi-Banach modulation spaces}



\begin{document}



\begin{abstract}
We prove the boundedness of a general class of multipliers and Fourier multipliers,
in particular of the Hilbert transform, on quasi-Banach modulation spaces.
We also deduce boundedness for multiplications and convolutions
for elements in such spaces.
\end{abstract}

\maketitle

\section{Introduction}\label{sec0}

\par

In the paper we deduce mapping properties of 
step multipliers and Fourier step multipliers
when acting on quasi-Banach modulation spaces.
Some parts of our investigations are based on
certain continuity properties for
multiplications and convolutions for elements in such spaces,
deduced in Section \ref{sec3}, and
which might be of independent interests.

\par

The Hilbert transform, i.{\,}e. multiplication by
the signum function on the Fourier transform side,  is frequently used in
mathematics, science and
technology. In physics it can be used to secure causality. For example,
in optics, the refractive index of a material is the frequency response of
a causal system whose real part gives the phase shift of the penetrating light
and the imaginary part gives the attenuation. The relationship between
the two are given by the Hilbert transform. Consequently, knowledge of one is
sufficient to retrieve the other.

\par

An inconveniently property with the Hilbert transform concerns lack of continuity
when acting on commonly used spaces. For example, 
it is well-known that the Hilbert transform is continuous on $L^2$, but fails to be
continuous on $L^p$ for $p\neq 2$ as well as on $\mascS$.
(See \cite{Ho1} and Section \ref{sec1} for notations.) 
A pioneering contribution which drastically improve the situation concerns
\cite{Oko1}, where K. Okoudjou
already in his thesis showed that the Hilbert transform is continuous
on the modulation space $M^{p,q}$ when $p\in (1,\infty)$ and $q\in [1,\infty ]$.
The result is surprising because $M^{p,q}$ is rather close to $L^p$ when
$q$ stays between $p$ and $p'$ (see e.{\,}g. \cite{Fei1,Toft2}).

\par

Okoudjou's result in \cite{Oko1} was extended in \cite{BenGraGroOko},
where B{\'e}nyi, Grafakos, Gr{\"o}chenig and Okoudjou
show that Fourier step multipliers, i.{\,}e. Fourier multipliers of the form
\begin{equation}\label{Eq:DefFourMultStepFunc}
f\mapsto \mascF ^{-1}
\left (\sum _{j\in b\mathbf Z}a_0(j)\chi _{j+[0,b)}\widehat f \right ),
\qquad a_0\in \ell ^\infty (b\mathbf Z),
\end{equation}
are continuous on the modulation space $M^{p,q}(\rr d)$, when
$p\in (1,\infty )$ and $q\in [1,\infty ]$. (See \cite[Theorem 1]{BenGraGroOko}.)

\par

Recall that
modulation spaces is a family of function and distribution spaces
introduced by Feichtinger in \cite{Fei1} and further developed
by Feichtinger and G{\"o}chenig in \cite{Fei6,FeiGro1,FeiGro2,FeiGro3,Gro2}.
In particular, the modulation spaces $M^{p,q}_{(\omega )}(\rr d)$ and
$W^{p,q}_{(\omega )}(\rr d)$ are
the set of tempered (or Gelfand-Shilov) distributions
whose short-time Fourier transforms belong to the weighted and mixed
Lebesgue spaces $L^{p,q}_{(\omega )}(\rr {2d})$ respectively
$L^{p,q}_{*,(\omega )}(\rr {2d})$.
Here $\omega$ is a weight function on phase (or time-frequency shift) space
and $p,q\in (0,\infty ]$. Note that $W^{p,q}_{(\omega )}(\rr d)$
is also an example on Wiener-amalgam spaces (cf. \cite{Fei6}).

\par

There are several convenient characterizations of modulation spaces. For
example, in \cite{Fei3,GaSa,Gro2,Gro2.5}, it
is shown that modulation spaces admit
reconstructible sequence space representations using Gabor frames.

\par

In Section \ref{sec2} we extend \cite[Theorem 1]{BenGraGroOko}
in several ways (see Theorems \ref{Thm:Mainthm1} and \ref{Thm:Mainthm2}).
\begin{enumerate}
\item The condition $q\in [1,\infty ]$ is relaxed into $q\in (0,\infty ]$.

\vrum

\item We allow weighted modulation spaces $M^{p,q}_{(\omega )}(\rr d)$,
where the weight $\omega$ only depends on the momentum or frequency
variable $\xi$, i.{\,}e. $\omega (x,\xi ) = \omega (\xi )$. These weights are
allowed to grow or decay at infinity, faster than polynomial growth.

\vrum

\item Our analysis also include continuity properties for the modulation
spaces $W^{p,q}_{(\omega )}(\rr d)$.
\end{enumerate}

\par

In similar ways as in \cite{BenGraGroOko}, we use Gabor analysis
for modulation spaces to show these properties. In \cite{BenGraGroOko},
the continuity for Fourier step multipliers are obtained by a convenient
choice of Gabor atoms in terms of Fourier transforms of second order
B-splines. This essentially transfer the critical continuity questions to
a finite set of discrete convolution operators acting on $\ell ^p$,
with dominating operator being the discrete Hilbert transform.
The choice of Gabor atoms then admit precise estimates of the
appeared convolution operators.

\par

In our situation the B-splines above are insufficient, because
B-splines lack in regularity, and when $p$ approaches $0$,
unbounded regularity
on the Fourier transform of the Gabor atoms are required. 
In fact, in order to obtain continuity for
weighted modulation spaces
with general moderate weights in the momentum variables,
it is required that the Fourier transform of Gabor atoms
obey even stronger regularities of Gevrey types.

\par

In Section \ref{sec4} we obtain some further extensions
and deduce precise estimates of the Fourier multipliers
in \eqref{Eq:DefFourMultStepFunc}, where more restrictive
$a_0$ should belong to $\ell ^{q}(b\mathbf Z)$ for some
$q\in (0,\infty ]$. In the end we are able to prove that
the Fourier multiplier in \eqref{Eq:DefFourMultStepFunc}
is continuous from $M^{p,q_1}$ to $M^{p,q_2}$ when
$p\in (1,\infty )$ and $q_1,q_2\in (0,\infty ]$ satisfy
$$
\frac 1{q_2}-\frac 1{q_1}\le \frac 1q.
$$

\par

More generally, in Section \ref{sec4} we generalize the continuity properties for
the step and Fourier step multiplier results in Section \ref{sec2} with
more general \emph{slope step} multiplier and Fourier slope step multipliers.
\\[3ex]
$\phantom k$
\noindent
\begin{tikzpicture}
\draw[->] (-0.25,0) -- (4.5,0) ;
\draw[->] (0,-0.25) -- (0,2.5) ;
\draw [smooth, samples=100,domain=0:0.76,variable=\x] plot(\x ,0.5 );
\draw [fill] (0,0.5) circle (1.1pt); 
\draw (0.8,0.5) circle (1.3pt); 
\draw [smooth, samples=100,domain=0.8:1.56,variable=\x] plot(\x ,{1});
\draw [fill] (0.8,1) circle (1.1pt); 
\draw (1.6,1) circle (1.3pt); 
\draw [smooth, samples=100,domain=1.6:2.36,variable=\x] plot(\x ,2.3 );
\draw [fill] (1.6,2.3) circle (1.1pt); 
\draw (2.4,2.3) circle (1.3pt); 
\draw [smooth, samples=100,domain=2.4:3.16,variable=\x] plot(\x ,0.8 );
\draw [fill] (2.4,0.8) circle (1.1pt); 
\draw (3.2,0.8) circle (1.3pt); 
\draw [smooth, samples=100,domain=3.2:3.96,variable=\x] plot(\x ,1.5 );
\draw [fill] (3.2,1.5) circle (1.1pt); 
\draw (4,1.5) circle (1.3pt); 
\draw[->] (6.25,0) -- (11,0) ;
\draw[->] (6.5,-0.25) -- (6.5,2.5) ;
\draw [smooth, samples=100,domain=6.5:7.28,variable=\x]
plot(\x ,{0.5+1.3*(\x -6.5)^2} );
\draw [fill] (6.5,0.5) circle (1.1pt); 
\draw (7.3,1.33) circle (1.3pt); 
\draw [smooth, samples=100,domain=7.3:8.09,variable=\x]
plot(\x ,{1+0.25*sin(16*\x r)});
\draw [fill] (7.3,0.867) circle (1.1pt); 
\draw       (8.1,0.822) circle (1.3pt); 
\draw [smooth, samples=100,domain=8.1:8.867,variable=\x] plot(\x ,2.3 );
\draw [fill] (8.1,2.3) circle (1.1pt); 
\draw (8.9,2.3) circle (1.3pt); 
\draw [smooth, samples=100,domain=8.9:9.665,variable=\x]
plot(\x ,{0.8+0.5*(9.7-\x )^2 } );
\draw [fill] (8.9,1.12) circle (1.1pt); 
\draw       (9.7,0.8) circle (1.3pt);
%
\draw [smooth, samples=100,domain=9.7:10.47,variable=\x]
plot(\x ,{\x -8.5} );
\draw [fill] (9.7,1.2) circle (1.1pt); 
\draw       (10.5,2)  circle (1.3pt);
\end{tikzpicture}

\par

Multiplier functions in Section \ref{sec2}.
\hspace{0.6cm}
Multiplier functions in Section \ref{sec4}.

\vspace{0.2cm}

An important ingredient for the proofs of the latter extension is multiplication
and convolution properties for $M^{p,q}_{(\omega )}$ and
$W^{p,q}_{(\omega )}$ spaces, given in Section \ref{sec3}.


\par

\begin{prop}\label{Prop:IntroMultConvMod}
Let $p_j,q_j\in (0,\infty ]$, $j=0,1,2$,
$$
\theta _1=\max \left (1 ,\frac 1{p_0},\frac 1{q_1},\frac 1{q_2} \right )
\quad \text{and}\quad
\theta _2=\max \left (1 ,\frac 1{p_1},\frac 1{p_2} \right ).
$$
Then
\begin{alignat*}{3}
M^{p_1,q_1}\cdot M^{p_2,q_2} &\subseteq M^{p_0,q_0},&
\qquad
\frac 1{p_1}+\frac 1{p_2} &=\frac 1{p_0}, &
\quad
\frac 1{q_1}+\frac 1{q_2} &= \theta _1 + \frac 1{q_0},
\\[1ex]
M^{p_1,q_1}* M^{p_2,q_2} &\subseteq M^{p_0,q_0}, &
\qquad
\frac 1{p_1}+\frac 1{p_2} &= \theta _2+\frac 1{p_0}, &
\quad
\frac 1{q_1}+\frac 1{q_2} &= \frac 1{q_0}.
\end{alignat*}
\end{prop}

\par

Similar result holds for $W^{p,q}$ spaces. The general multiplication
and convolution properties in Section \ref{sec3} also overlap with
results by Bastianoni, Cordero and Nicola in \cite{BaCoNi},
by Bastianoni and Teofanov in \cite{BaTe},
and by Guo, Chen, Fan and Zhao in \cite{GuChFaZh}.

\par

The multiplication
relation in Proposition \ref{Prop:IntroMultConvMod} for $p_j,q_j\ge 1$
was obtained already in \cite{Fei1} by Feichtinger. It is also obvious that
the convolution relation was well-known since then
(though a first formal proof of this relation seems to be given first in
\cite{Toft3}). In general, these convolution and multiplication
properties follow the rules
\begin{alignat*}{4}
\ell ^{p_1}&*\ell ^{p_2}\subseteq \ell ^{p_0}, & \quad 
\ell ^{q_1} &\cdot \ell ^{q_2}\subseteq \ell ^{q_0} &
\quad &\Rightarrow & \quad
M^{p_1,q_1}&*M^{p_2,q_2}\subseteq M^{p_0,q_0}
\intertext{and}
\ell ^{p_1} &\cdot \ell ^{p_2}\subseteq \ell ^{p_0}, & \quad 
\ell ^{q_1} &* \ell ^{q_2}\subseteq \ell ^{q_0} &
\quad &\Rightarrow & \quad
M^{p_1,q_1} &\cdot M^{p_2,q_2}\subseteq M^{p_0,q_0},
\end{alignat*}
which goes back to \cite{Fei1} in the Banach space case and to
\cite{GaSa} in the quasi-Banach case. See also \cite{FeiGro1}
and \cite{Rau2}
for extensions of these relations to more general Banach function
spaces and quasi-Banach function spaces, respectively.

\par

In Section \ref{sec3} we extend the multiplication and
convolution results in \cite{BaCoNi,BaTe,GuChFaZh} to allow
more general weights as well as finding multi-linear versions.
We stress that the results in Section \ref{sec3} hold true for general
moderate weights, while corresponding results in \cite{GuChFaZh}
are formulated only for polynomially moderate weights which also
should be split, i.{\,}e. of the form $\omega (x,\xi )=\omega _1(x)\omega _2(\xi)$.
In Section \ref{sec3} we also carry out questions on uniqueness for
extensions of multiplications and convolutions from
the Gelfand-Shilov space $\Sigma _1(\rr d)$,
to the involved modulation spaces. Note that $\Sigma _1(\rr d)$
is dense in $\mascS (\rr d)$ and is contained in all
modulation spaces with moderate weights (see e.{\,}g. \cite{Toft10}).
On the other hand, in contrast to \cite{GuChFaZh},
we do not deduce any sharpness for our results.

\par

%
%

%

\par

The analysis to show Proposition \ref{Prop:IntroMultConvMod}
is more complex compared to the restricted case when $p_j,q_j\ge 1$,
because of absence of local-convexity of involved spaces when
some of the Lebesgue exponents are smaller than one.
In fact, the desired estimates when $p_j,q_j\ge 1$
can be achieved by straight-forward applications of H{\"o}lder's and
Young's inequalities. For corresponding estimates
in Proposition \ref{Prop:IntroMultConvMod}$'$, some additional
arguments seems to be needed. In our situation we
discretize the situations in similar ways as in \cite{BaCoNi} by using
Gabor analysis for modulation spaces, and then apply some
further arguments, valid in non-convex analysis. This
approach is slightly different compared to what is used in
\cite{GuChFaZh} which follows the discretization technique
introduced in \cite{WaHu}, and which has some traces of Gabor analysis.

\par

A non-trivial question concerns wether the multiplications
and convolutions in Propositions \ref{Prop:IntroMultConvMod}
and \ref{Prop:IntroMultConvMod}$'$ are uniquely defined or not.
If $p_j,q_j<\infty$, $j=1,2$, then the uniqueness is evident because
the Schwartz space is dense in $M^{p_j,q_j}$. In the case
$p_1,q_1<\infty$ or $p_2,q_2<\infty$, the uniqueness in Proposition
\ref{Prop:IntroMultConvMod} follows from the first case,
duality and embedding properties for quasi-Banach modulation spaces
into Banach modulation spaces.
The uniqueness in \ref{Prop:IntroMultConvMod}$'$ then
follows from the uniqueness in Proposition \ref{Prop:IntroMultConvMod}
and the fact that $M^{p,q}$ increases with $p$ and $q$.

\par

A critical situation appear when $p_1+q_1=p_2+q_2=\infty$. Then
$\mascS$ is neither dense in $M^{p_1,q_1}$ nor in
$M^{p_2,q_2}$. For the multiplications in Propositions
\ref{Prop:IntroMultConvMod},
the uniqueness can be obtained by suitable approaches based on
the so-called narrow convergence, which is a weaker form of
convergence compared to norm convergence (see \cite{Sjo,Toft2,Toft10}).
However, for the
convolution in Propositions \ref{Prop:IntroMultConvMod}, we are not able
to show any uniqueness of these extensions in this critical situation.

\par

The paper is organized as follows. In Section \ref{sec1} we present well-known
properties of Gelfand-Shilov spaces, modulation spaces, multipliers and
Fourier multipliers. In Section \ref{sec2} we deduce continuity
properties for step and Fourier step multipliers when acting on (quasi-Banach)
modulation spaces. Then we establish convolution and continuity properties
for quasi-Banach modulation spaces in Section \ref{sec3}.
In Section \ref{sec4} we show how the multiplication and convolution results
in Section \ref{sec3} can be used to generalize the continuity results
in Section \ref{sec2}, to
more general slope step multiplier and Fourier slope step multipliers.
Finally we present a proof of a multi-linear convolution result in Appendix
\ref{AppA}. 

\par

\section*{Acknowledgement}

\par

The idea of the paper appeared when I supervised
Nils Zandler-Andersson for his bachelor degree (see \cite{Zan}). In those
thesis, Mr. Andersson deduced some extensions of
the multiplier results in \cite{BenGraGroOko} to certain quasi-Banach
modulation spaces (see \cite[Theorem 4.16]{Zan}). I am also
grateful to Elena Cordero and Nenad Teofanov for reading the
paper and giving valuable comments, leading to improvements of the
content.

\par

\section{Preliminaries}\label{sec1}

\par

In this section we recall some facts on Gelfand-Shilov spaces,
modulation spaces, discrete convolutions, step and Fourier step
multipliers. After explaining some properties of the Gelfand-Shilov
spaces and their distribution spaces, we consider a suitable
twisted convolution and recall some facts on weight functions and
mixed norm spaces.
Thereafter we consider classical modulation spaces,
which are more general compared Feichtinger in \cite{Fei1}
in the sense of more general weights as well as we permit the
Lebesgue exponents to belong to the full interval $(0,\infty ]$
instead of $[1,\infty ]$. Here we also recall some facts on
Gabor expansions for modulation spaces. Then we collect some
facts on discrete convolution estimates on weighted $\ell ^p$
spaces with the exponents in the full interval $(0,\infty ]$.
We finish the section by giving the definition of step and Fourier step
multipliers. 

\par

\subsection{Gelfand-Shilov spaces and their distribution spaces}

\par

For any $0<h,s,\sigma \in \mathbf R$, $\mathcal S_{s,h}^\sigma (\rr d)$
consists of all
$f\in C^\infty (\rr d)$ such that
\begin{equation}\label{gfseminorm}
\nm f{\mathcal S_{s,h}^\sigma}\equiv \sup \frac {|x^\beta \partial ^\alpha
f(x)|}{h^{|\alpha  + \beta |}\alpha !^\sigma \beta !^s}
\end{equation}
is finite. Then $\mathcal S_{s,h}^\sigma (\rr d)$ is a Banach space with norm
$\nm \cdo{\mathcal S_{s,h}^\sigma}$.
The \emph{Gelfand-Shilov spaces} $\mathcal S_s^\sigma (\rr d)$ 
and $\Sigma _s^\sigma (\rr d)$, of Roumieu and Beurling 
types respectively, are the inductive and projective limits of
$\mathcal S_{s,h}^\sigma (\rr d)$ with respect to $h>0$ (see e.{\,}g.
\cite{GeSh}). It follows that
\begin{equation}\label{GSspacecond1}
\mathcal S_s^\sigma (\rr d) = \bigcup _{h>0}\mathcal S_{s,h}^\sigma (\rr d)
\quad \text{and}\quad \Sigma _s^\sigma (\rr d) =\bigcap _{h>0}
\mathcal S_{s,h}^\sigma (\rr d)
\end{equation}
We remark that 
$\Sigma _s^\sigma (\rr d)\neq \{ 0\}$, if and only if $s+\sigma > 1$, and
$\maclS _s^\sigma (\rr d)\neq \{ 0\}$, if and only
if $s+\sigma \ge 1$, and that
$$
\maclS _{s_1}^{\sigma _1}(\rr d)\subseteq \Sigma _{s_2}^{\sigma _2}(\rr d)
\subseteq \maclS _{s_2}^{\sigma _2}(\rr d)
\subseteq \mascS (\rr d),
\qquad s_1<s_2 ,\ \sigma _1<\sigma _2.
$$

\par

The \emph{Gelfand-Shilov distribution spaces} $(\mathcal S_s^\sigma )'(\rr d)$ 
and $(\Sigma _s^\sigma )'(\rr d)$, of Roumieu and Beurling 
types respectively, are the (strong) duals of $\mathcal S_s^\sigma (\rr d)$ 
and $\Sigma _s^\sigma (\rr d)$, respectively. It follows that if
$(\mathcal S_{s,h}^\sigma )'(\rr d)$ is the $L^2$-dual of
$\mathcal S_{s,h}^\sigma (\rr d)$ and $s+\sigma \ge 1$
($s+\sigma > 1$),
then $(\mathcal S_s^\sigma )'(\rr d)$
($(\Sigma _s^\sigma )'(\rr d)$) can be identified
with the projective limit (inductive limit) of
$(\mathcal S_{s,h}^\sigma )'(\rr d)$ with respect to $h>0$. It follows that
\begin{equation}\label{GSspacecond2}
(\mathcal S_s^\sigma )'(\rr d) = \bigcap _{h>0}(\mathcal S_{s,h}^\sigma )'(\rr d)
\quad \text{and}\quad \Sigma _s'(\rr d) =\bigcup _{h>0}
(\mathcal S_{s,h}^\sigma )'(\rr d)
\end{equation}
for such choices of $s$ and $\sigma$. (See \cite{Pil}.) We remark that 
\begin{align*}
\mascS '(\rr d) &\subseteq (\maclS _{s_2}^{\sigma _2})'(\rr d)\subseteq
(\Sigma _{s_2}^{\sigma _2})'(\rr d)\subseteq (\maclS _{s_1}^{\sigma _1})'(\rr d),
\intertext{when}
s_1 &< s_2 ,\ \sigma _1<\sigma _2\quad \text{and}\quad s_1+\sigma _1\ge 1.
\end{align*}
For convenience we set $\maclS _s=\maclS _s^s$ and
$\Sigma _s=\Sigma _s^s$.

\par

The Gelfand-Shilov spaces are invariant under several basic transformations.
For example they are invariant under translations, dilations
and under (partial) Fourier transformations.
%
%
In fact, let $\mathscr F$ be the Fourier transform which takes the form
$$
(\mathscr Ff)(\xi )= \widehat f(\xi ) \equiv (2\pi )^{-\frac d2}\int _{\rr
{d}} f(x)e^{-i\scal  x\xi }\, dx
$$
when $f\in L^1(\rr d)$. Here $\scal \cdo \cdo$ denotes the usual scalar product
on $\rr d$. The map $\mathscr F$ extends 
uniquely to homeomorphisms on $\mathscr S'(\rr d)$, from
$(\mathcal S_s^\sigma )'(\rr d)$ to $(\mathcal S_\sigma ^s )'(\rr d)$
and from $(\Sigma _s^\sigma )'(\rr d)$ to $(\Sigma _\sigma ^s )'(\rr d)$.
Then the map $\mathscr F$ restricts
to homeomorphisms on $\mathscr S(\rr d)$, from
$\mathcal S_s^\sigma (\rr d)$ to $\mathcal S_\sigma ^s (\rr d)$,
from $\Sigma _s^\sigma (\rr d)$ to $\Sigma _\sigma ^s (\rr d)$,
and to a unitary operator on $L^2(\rr d)$.

\par

There are several characterizations of Gelfand-Shilov spaces and their 
distribution spaces (cf. \cite{ChuChuKim,Eij,Toft18} and the references therein).
For example, it follows from \cite{ChuChuKim,Eij} that the following is true.
Here $g(\theta )\lesssim h(\theta )$,
$\theta \in \Omega$, means that there is a constant $c>0$ such that
$g(\theta )\le ch(\theta )$ for all $\theta \in \Omega$.

\par

\begin{prop}\label{Prop:GSexpcond}
Let $f\in \mascS '(\rr d)$ and $s,\sigma >0$.
Then the following conditions are equivalent:
\begin{enumerate}
\item $f\in \maclS _s^\sigma (\rr d)$ ($f\in \Sigma _s^\sigma (\rr d)$);

\vrum

\item $|f(x)|\lesssim e^{-r|x|^{\frac 1s}}$ and $|\widehat f (\xi )|
\lesssim e^{-r |\xi |^{\frac 1\sigma}}$ for some $r>0$ (for every $r>0$);

\vrum

\item $f\in C^\infty (\rr d)$ and $|(\partial ^\alpha f)(x)|
\lesssim h^{|\alpha |}\alpha !^\sigma e^{-r|x|^{\frac 1s}}$
for some $r,h>0$ (for every $r,h>0$).
\end{enumerate}
\end{prop}

\par

Gelfand-Shilov spaces and their distribution spaces can also be
characterized by estimates on their short-time Fourier transforms
Let $\phi \in \maclS _s(\rr d)$ ($\phi \in \Sigma _s(\rr d)$) be fixed.
Then the short-time Fourier transform of $f\in \maclS _s'(\rr d)$
(of $f\in \Sigma _s'(\rr d)$) with respect to $\phi$ is defined by
\begin{align}
(V_\phi f)(x,\xi ) &\equiv (2\pi )^{-\frac d2}(f,\phi (\cdo -x)e^{i\scal \cdo \xi})_{L^2}.
\label{Eq:STFTDef}
\intertext{We observe that}
(V_\phi f)(x,\xi ) &= \mascF (f\cdot \overline {\phi (\cdo -x)})(\xi )
\tag*{(\ref{Eq:STFTDef})$'$}
\intertext{(cf. \cite{Toft22}). If in addition $f\in L^p(\rr d)$ for some $p\in [1,\infty ]$, then}
(V_\phi f)(x,\xi ) &= (2\pi )^{-\frac d2}\int _{\rr d}f(y)\overline {\phi (y-x)}e^{-i\scal y\xi}\, dy.
\tag*{(\ref{Eq:STFTDef})$''$}
\end{align}

\par

In the next lemma we present characterizations of
Gelfand-Shilov spaces and their distribution spaces
in terms of estimates on the short-time Fourier transforms of the
involved elements. The proof is omitted, since
the first part follows from \cite{GroZim}, and the second part from
\cite{Toft10,Toft18}.

\par

\begin{lemma}\label{Lemma:GSFourierest}
Let $p\in [1,\infty ]$, $f\in \mathcal S'_{1/2}(\rr d)$, $s,\sigma >0$,
$\phi \in \maclS _s^\sigma (\rr d)\setminus 0$
($\phi \in \Sigma _s^\sigma (\rr d)\setminus 0$)
and
$$
v_r(x,\xi ) = e^{r(|x|^{\frac 1s}+|\xi |^{\frac 1\sigma})},\qquad r\ge 0.
$$
Then the following is true:
\begin{enumerate}
\item[(1)] $f\in \mathcal S_s^\sigma (\rr d)$ ($f\in \Sigma _s^\sigma (\rr d)$),
if and only if
\begin{equation}\label{Eq:GSExpCond1}
\nm {V_\phi f\cdot v_r}{L^p}< \infty
\end{equation}
for some $r>0$ (for every $r>0$);

\vrum

\item[(2)] $f\in (\mathcal S_s^\sigma )'(\rr d)$
($f\in (\Sigma _s^\sigma )'(\rr d)$), if and only if
\begin{equation}\label{Eq:GSExpCond2}
\nm {V_\phi f/ v_r}{L^p}< \infty
\end{equation}
for every $r>0$ (for some $r>0$).
\end{enumerate}
\end{lemma}

%

\par

We also need the following. Here the first part is a straight-forward consequence
of the definitions, and the second part follows from the first part and duality.
The details are left for the reader.

\par

\begin{prop}\label{Prop:AdjointSTFTCont}
Let $\phi \in \Sigma _s(\rr d)\setminus 0$. Then the following is true:
\begin{enumerate}
\item $V_\phi$ is continuous from $\Sigma _s(\rr d)$ to
$\Sigma _s(\rr {2d})$ and from $\Sigma _s'(\rr d)$ to $\Sigma _s'(\rr {2d})$;

\vrum

\item $V_\phi ^*$ is continuous from $\Sigma _s(\rr {2d})$ to
$\Sigma _s(\rr d)$ and from $\Sigma _s'(\rr {2d})$ to $\Sigma _s'(\rr d)$.
\end{enumerate}

\par

The same holds true with $\maclS _s$ or $\mascS$ in place of $\Sigma _s$
at each occurrence.
\end{prop}

\par

\subsection{A suitable twisted convolution}\label{subsec1.5?}

\par

Let $f$ be a distribution on $\rr d$, $\phi ,\phi _j$, $j=1,2,3$, be suitable
test functions on $\rr d$, and let $F$ and $G$ be a pair of suitable
distribution/test function on $\rr {2d}$. Then the
\emph{twisted convolution} $F*_VG$ of $F$ and $G$ is defined by
\begin{align}
(F*_VG)(x,\xi )
&=
(2\pi )^{-\frac d2}
\iint _{\rr {2d}} \! \!
F(x-y,\xi -\eta )G(y,\eta )e^{-i\scal{y}{\xi -\eta}}\, dyd\eta .
\notag
\\[1ex]
&=
(2\pi )^{-\frac d2}
\iint _{\rr {2d}} \! \!
F(y,\eta )G(x-y,\xi - \eta )e^{-i\scal{x-y}\eta}\, dyd\eta .
\label{Eq:TwistConvDef}
\end{align}
The convolution above should be interpreted as
\begin{align}
(F*_VG)(X) &= (2\pi )^{-\frac d2}\scal {F(X-\cdo )e^{-i\Phi (X,\cdo )}}{G}
\notag
\\[1ex]
 &= (2\pi )^{-\frac d2}\scal {F}{G(X-\cdo )e^{-i\Phi (X,X-\cdo )}},
\tag*{(\ref{Eq:TwistConvDef})$'$}
\\[1ex]
\text{where}\quad
\Phi (X,Y) &= \scal y{\xi -\eta} ,
\quad
X=(x,\xi )\in \rr {2d},\ Y=(y,\eta )\in \rr {2d},
\notag
\end{align}
when $F$ belongs to a distribution space on $\rr {2d}$ and $G$ belongs
to the corresponding test function space. By straight-forward computations
it follows that 
\begin{equation}\label{Eq:TwistedConvAsoc}
(F*_VG)*_VH = F*_V(G*_VH),
\end{equation}
when $F$, $H$ are distributions and $G$ is a test function, or
$F$, $H$ are test functions and $G$ is a distribution.

\par

\begin{rem}\label{Rem:BasicPropTwistConv}
Let $s>0$. An important property of $*_V$ above is that if
$f\in \Sigma _s'(\rr d)$ and $\phi _j\in \Sigma _s(\rr d)$
and $\phi \in \Sigma _s(\rr d)\setminus 0$, $j=1,2,3$, then
it follows by straight-forward applications of Parseval's formula that
\begin{align}
\big ( (V_{\phi _2}\phi _3) *_V(V_{\phi _1}f)  \big ) (x,\xi )
&=
(\phi _3,\phi _1)_{L^2} \cdot (V_{\phi _2}f)(x,\xi ).
\label{Eq:STFTWindTrans}
\intertext{and that if}
P_\phi &\equiv \nm \phi {L^2}^{-2}\cdot V_\phi \circ V_\phi ^*,
\label{Eq:ProjphiDef}
\intertext{then}
P_\phi F &= \nm \phi {L^2}^{-2}\cdot V_\phi \phi *_V F 
\intertext{when $F\in \Sigma _s'(\rr {2d})$. We observe that}
P_\phi ^* &= P_\phi
\quad \text{and}\quad P_\phi ^2=P_\phi .
\label{Eq:ProjphiRule}
\end{align}

(See e.{\,}g. Chapters 11 and 12 in \cite{Gro2}.)

\par

We also remark that if $F\in \Sigma _s'(\rr {2d})$, then $F=V_{\phi}f$ for some
$f\in \Sigma _s'(\rr {d})$, if and only if
\begin{equation}\label{Eq:TwistedProj}
F= P_\phi F.
\end{equation}
Furthermore, if \eqref{Eq:TwistedProj} holds, then $F=V_{\phi}f$ with
\begin{equation}\label{Eq:TwistedProj2}
f=\nm \phi{L^2}^{-2}V_\phi ^*F .
\end{equation}

\par

In fact, suppose that $f\in \Sigma _s'(\rr d)$ and let $F=V_\phi f$. Then
\eqref{Eq:TwistedProj} follows from \eqref{Eq:STFTWindTrans}.

\par

On the other hand, suppose that \eqref{Eq:TwistedProj} holds and let $f$
be given by \eqref{Eq:TwistedProj2}. Then
$$
V_\phi f = P_\phi F= F,
$$
and the asserted equivalence follows.

\par

We notice that the same holds true with $\maclS _s$ or $\mascS$ in place of $\Sigma _s$
at each occurrence.
\end{rem}

\par

\subsection{Mixed norm space of Lebesgue types}

\par

A weight on $\rr d$ is a function $\omega _0\in L^\infty _{loc}(\rr d)$ such that
$1/\omega _0\in L^\infty _{loc}(\rr d)$. The weight $\omega _0$ on $\rr d$ is called
\emph{moderate}, if there is an other weight $v$ on $\rr d$ such that
\begin{equation}\label{Eq:ModCond}
\omega (x+y)\lesssim \omega (x)v(y),\qquad x,y\in \rr d.
\end{equation}
The set of moderate weights on $\rr d$ is denoted by $\mascP _E(\rr d)$, and
if $s>0$, then $\mascP _{E,s}(\rr d)$ is the set of all moderate weights $\omega _0$
on $\rr d$ such that \eqref{Eq:ModCond} holds for $v(y)=e^{r|y|^{\frac 1s}}$
for some $r>0$. We also let $\mascP _{E,s}^\sigma (\rr {2d})$ be the set of all weights
$\omega$ such that
$$
\omega (x+y,\xi +\eta )\lesssim \omega (x,y)e^{r(|y|^{\frac 1s}+|\eta |^{\frac 1\sigma})}
$$
for some $r>0$. We recall that if $\omega \in \mascP _E(\rr d)$, then
there is a constant $r\ge 0$ such that
$$
\omega (x+y)\lesssim \omega (x)e^{r|y|},\qquad x,y\in \rr d.
$$
In particular, $\mascP _{E,s}(\rr d)=\mascP _E(\rr d)$ when $s\le 1$
(see \cite{Gro3}).

\par

For any weight $\omega$ on $\rr {2d}$ and for every $p,q\in (0,\infty ]$,
we set
\begin{alignat*}{3}
\nm {F}{L^{p,q}_{(\omega )}(\rr {2d})} &\equiv \nm {G_{F,\omega ,p}}{L^q(\rr d)}, &
\quad &\text{where} & \quad
G_{F,\omega ,p}(\xi ) &= \nm {F(\cdo ,\xi )\omega (\cdo ,\xi )}{L^p(\rr d)}
\intertext{and}
\nm {F}{L^{p,q}_{*,(\omega )}(\rr {2d})} &\equiv \nm {H_{F,\omega ,q}}{L^p(\rr d)}, &
\quad &\text{where} & \quad
H_{F,\omega ,q}(x) &= \nm {F(x,\cdo )\omega (x,\cdo )}{L^q(\rr d)},
\end{alignat*}
when $F$ is (complex-valued) measurable function on $\rr {2d}$. Then
$L^{p,q}_{(\omega )}(\rr {2d})$ ($L^{p,q}_{*,(\omega )}(\rr {2d})$)
consists of all measurable functions
$F$ such that $\nm F{L^{p,q}_{(\omega )}}<\infty$
($\nm F{L^{p,q}_{*,(\omega )}}<\infty$).

\par

In similar ways, let $\Omega _1,\Omega _2$ be discrete sets and
$\ell _0'(\Omega _1\times \Omega _2)$ consists of all formal (complex-valued)
sequences $c=\{ c(j,k)\} _{j\in \Omega _1,k\in \Omega _2}$. Then
the discrete Lebesgue spaces
$$
\ell ^{p,q}_{(\omega )}(\Omega _1\times \Omega _2)
\quad \text{and}\quad
\ell ^{p,q}_{*,(\omega )}(\Omega _1\times \Omega _2)
$$
of mixed (quasi-)norm types consists of all $c\in \ell _0'(\Omega _1\times \Omega _2)$
such that $\nm {c}{\ell^{p,q}_{(\omega )}(\Omega _1\times \Omega _2)}<\infty$
respectively $\nm {c}{\ell^{p,q}_{*,(\omega )}(\Omega _1\times \Omega _2)}<\infty$.
Here
\begin{alignat*}{3}
\nm {c}{\ell ^{p,q}_{(\omega )}(\Omega _1\times \Omega _2)}
&\equiv \nm {G_{F,\omega ,p}}{\ell ^q(\Omega _2)}, &
\quad &\text{where} & \quad
G_{c,\omega ,p}(k) &= \nm {F(\cdo ,k )\omega (\cdo ,k )}{\ell ^p(\Omega _1)}
\intertext{and}
\nm {c}{L^{p,q}_{*,(\omega )}(\Omega _1\times \Omega _2)}
&\equiv \nm {H_{c,\omega ,q}}{\ell ^p(\Omega _1)}, &
\quad &\text{where} & \quad
H_{c,\omega ,q}(j) &= \nm {c(j,\cdo )\omega (j,\cdo )}{\ell ^q(\Omega _2)},
\end{alignat*}
when $c\in \ell _0'(\Omega _1\times \Omega _2)$.

\par

\subsection{Modulation spaces and other Wiener type spaces}

%
%

The (classical) modulation spaces, essentially introduced in
\cite{Fei1} by Feichtinger are given given in the following.
(See e.{\,}g. \cite{Fei6} for definition of more general modulation spaces.)

\par

\begin{defn}
Let $p,q\in (0,\infty ]$, $\omega \in \mascP _E(\rr {2d})$ and
$\phi \in \Sigma _1(\rr d)\setminus 0$.
\begin{enumerate}
\item The \emph{modulation space} $M^{p,q}_{(\omega )}(\rr d)$
consists of all $f\in \Sigma _1'(\rr d)$ such that
$$
\nm f{M^{p,q}_{(\omega )}}\equiv \nm {V_\phi f}{L^{p,q}_{(\omega )}}
$$
is finite. The topology of $M^{p,q}_{(\omega )}(\rr d)$ is defined by
the (quasi-)norm $\nm \cdo{M^{p,q}_{(\omega )}}$;

\vrum

\item The \emph{modulation space (of Wiener amalgam type)}
$W^{p,q}_{(\omega )}(\rr d)$ consists of all $f\in \Sigma _1'(\rr d)$ such that
$$
\nm f{W^{p,q}_{(\omega )}}\equiv \nm {V_\phi f}{L^{p,q}_{*,(\omega )}}
$$
is finite. The topology of $W^{p,q}_{(\omega )}(\rr d)$ is defined by
the (quasi-)norm $\nm \cdo{W^{p,q}_{(\omega )}}$.
\end{enumerate}
\end{defn}

\par

\begin{rem}\label{Rem:ModSpaces}
Modulation spaces possess several convenient properties.
In fact, let $p,q\in (0,\infty ]$, $\omega \in \mascP _E(\rr {2d})$ and
$\phi \in \Sigma _1(\rr d)\setminus 0$. Then the following is true
(see \cite{Fei1,Fei6,FeiGro1,FeiGro2,GaSa,Gro2} and
their analyses for verifications):
\begin{itemize}
\item the definitions of $M^{p,q}_{(\omega )}(\rr d)$ and $W^{p,q}_{(\omega )}(\rr d)$
are independent of the choices of $\phi \in \Sigma _1(\rr d)\setminus 0$,
and different choices give rise to equivalent quasi-norms;

\vrum

\item the spaces $M^{p,q}_{(\omega )}(\rr d)$ and $W^{p,q}_{(\omega )}(\rr d)$
are quasi-Banach spaces which increase with $p$ and $q$, and decrease with
$\omega$. If in addition $p,q\ge 1$, then they are Banach spaces.

\vrum

\item $\Sigma _1(\rr d)\subseteq M^{p,q}_{(\omega )}(\rr d),W^{p,q}_{(\omega )}(\rr d)
\subseteq \Sigma _1'(\rr d)$;

\vrum

\item If in addition $p,q\ge 1$, then the $L^2(\rr d)$ scalar product,
$(\cdo ,\cdo )_{L^2(\rr d)}$, on $\Sigma _1(\rr d)\times \Sigma (\rr d)$ is
uniquely extendable to dualities between $M^{p,q}_{(\omega)}(\rr d)$
and $M^{p',q'}_{(1/\omega)}(\rr d)$, and between $W^{p,q}_{(\omega)}(\rr d)$
and $W^{p',q'}_{(1/\omega)}(\rr d)$. If in addition $p,q<\infty$, then
the duals of $M^{p,q}_{(\omega)}(\rr d)$ and $W^{p,q}_{(\omega)}(\rr d)$
can be identified with $M^{p',q'}_{(1/\omega)}(\rr d)$ respectively
$W^{p',q'}_{(1/\omega)}(\rr d)$, through the form $(\cdo ,\cdo )_{L^2(\rr d)}$;

\vrum

\item Let $\omega _0(x,\xi )= \omega (-\xi ,x)$. Then
$\mascF$ on $\Sigma _1'(\rr d)$
restricts to a homeomorphism from $M^{p,q}_{(\omega )}(\rr d)$ to
$W^{q,p}_{(\omega _0)}(\rr d)$.
\end{itemize}
\end{rem}

\par

\subsection{Gabor expansions for modulation spaces}

\par

A fundamental property for modulation spaces is that 
they can be discretized in convenient ways by Gabor
expansions. For fundamental contributions, see e.{\,}g.
\cite{ChrKimKim,Fei3,FeiGro1,FeiGro2,FeiGro3,GaSa,Gro1,Gro2,GroZim}
and the references therein. Here we present a straight
way to obtain such expansions in the case when we may find
compactly supported Gabor atoms.

\par

Let $s,\sigma >0$ be such that $s+\sigma \ge 1$.
Then $\maclD ^\sigma (\rr d)$ is the set of
all compactly supported elements in
$\maclS _s^\sigma (\rr d)$. That is, $\maclD ^\sigma (\rr d)$
consists of all $\phi \in C_0^\infty (\rr d)$ such that
$$
\nm {\partial ^\alpha \phi}{L^\infty} \lesssim h^{|\alpha |}\alpha !^\sigma 
$$
holds true for some $h>0$. We recall that if $\sigma \le 1$, then
$\maclD ^\sigma (\rr d)$ is trivial (i.{\,}e. $\maclD ^\sigma (\rr d)=\{ 0\}$).
If instead $\sigma >1$, then $\maclD ^\sigma (\rr d)$ is dense in
$C_0^\infty (\rr d)$.

\par

From now on we suppose that $\sigma >1$, giving that
$\maclD ^\sigma (\rr d)$ is non-trivial. In view of Sections 1.3 and 1.4 in
\cite{Ho1}, we may find $\phi ,\psi \in \maclD ^\sigma (\rr d)$ with values in
$[0,1]$ such that
\begin{alignat}{4}
\supp \phi &\subseteq \Big [-\frac 34,\frac 34\Big ] ^d,&
\quad
\phi (x) &= 1&
\quad &\text{when}& \quad
x&\in \Big [-\frac 14,\frac 14\Big ]^d
\label{Eq:phiProp}
\\[1ex]
\supp \psi &\subseteq [-1,1]^d,&
\quad
\psi (x) &= 1&
\quad &\text{when}& \quad
x&\in \Big [-\frac 34,\frac 34\Big ]^d
\label{Eq:psiProp}
\end{alignat}
and
\begin{equation}\label{Eq:PartUnity1}
\sum _{j\in \zz d}\phi (\cdo -j) =1.
\end{equation}

\par

Let $f\in (\maclS _s^\sigma )'(\rr d)$. Then
$x\mapsto f(x)\phi (x-j)$ belongs to $(\maclS _s^\sigma )'(\rr d)$ and is supported
in $j+[-\frac 34,\frac 34]^d$. Hence, by periodization it follows from Fourier analysis
that
\begin{equation}\label{Eq:FirstExp}
f(x)\phi (x-j) = \sum _{\iota \in \pi \zz d}c(j,\iota )e^{i\scal x\iota},
\qquad
x\in j+[-1,1]^d,
\end{equation}
where
$$
c(j,\iota ) = 2^{-d}(f,\phi (\cdo -j)e^{i\scal \cdo \iota})
=
\left ( \frac \pi 2\right )^{\frac d2} V_\phi f(j,\iota ),
\qquad j\in \zz d,\ \iota \in \pi \zz d .
$$
Since $\psi =1$ on the support of $\phi$, \eqref{Eq:FirstExp} gives
\begin{equation}\tag*{(\ref{Eq:FirstExp})$'$}
f(x)\phi (x-j) = \left ( \frac \pi 2\right )^{\frac d2}
\sum _{\iota \in \pi \zz d}V_\phi f(j,\iota )\psi (x-j)e^{i\scal x\iota},
\qquad
x\in \rr d,
\end{equation}
By \eqref{Eq:PartUnity1} it now follows that
\begin{align}
f(x) &= \left ( \frac \pi 2\right )^{\frac d2}
\sum _{(j,\iota )\in \Lambda}V_\phi f(j,\iota )\psi (x-j)e^{i\scal x\iota},
\qquad
x\in \rr d,
\label{Eq:SpecGaborExp1}
\intertext{where}
\Lambda &= \zz d\times (\pi \zz d),
\label{Eq:OurLattice}
\end{align}
which is the \emph{Gabor expansion} of $f$ with respect to the
\emph{Gabor pair} $(\phi ,\psi )$ and lattice $\Lambda$,
i.{\,}e. with respect to the \emph{Gabor atom}
$\phi$ and the \emph{dual Gabor atom} $\psi$.
Here the series converges in
$(\maclS _s^\sigma )'(\rr d)$. By duality and the fact that
$\maclD ^\sigma (\rr d)$ is dense in $(\maclS _s^\sigma )'(\rr d)$
we also have
\begin{equation}\label{Eq:SpecGaborExp2}
f(x) = \left ( \frac \pi 2\right )^{\frac d2}
\sum _{(j,\iota )\in \Lambda}V_\psi f(j,\iota )\phi (x-j)e^{i\scal x\iota},
\qquad
x\in \rr d,
\end{equation}
with convergence in $(\maclS _s^\sigma )'(\rr d)$.

\par

Let $T$ be a linear continuous operator from $\maclS _s^\sigma (\rr d)$ to
$(\maclS _s^\sigma )'(\rr d)$ and let $f\in \maclS _s^\sigma (\rr d)$. Then
it follows from \eqref{Eq:SpecGaborExp1} that
$$
(Tf)(x) = \left ( \frac \pi 2\right )^{\frac d2}
\sum _{(j,\iota )\in \Lambda}V_\phi f(j,\iota )T(\psi (\cdo -j)e^{i\scal \cdo \iota})(x)
$$
and
$$
T(\psi (\cdo -j)e^{i\scal \cdo \iota})(x)
=
\left ( \frac \pi 2\right )^{\frac d2}
\sum _{(k,\kappa )\in \Lambda}(V_\phi (T(\psi (\cdo -j)e^{i\scal \cdo \iota})))(k,\kappa )
\psi (x-k)e^{i\scal x\kappa}.
$$
A combination of these expansions show that
\begin{equation}\label{Eq:SpecGaborExp3}
(Tf)(x) = \left ( \frac \pi 2\right )^{\frac d2}
\sum _{(j,\iota )\in \Lambda} (A\cdot V_\phi f)(j,\iota )\psi (x-j)e^{i\scal x\iota},
\end{equation}
where $A=(a(\mabfj ,K))_{\mabfj ,\mabfk \in \Lambda}$ is the
$\Lambda \times \Lambda$-matrix, given by
\begin{multline}\label{Eq:OpGaborMatrix}
a(\mabfj ,\mabfk ) = \left ( \frac \pi 2\right )^{\frac d2}
(T(\psi (\cdo -j)e^{i\scal \cdo \iota}) , \phi (\cdo -k)e^{i\scal \cdo \kappa} )_{L^2(\rr d)}
\\[1ex]
\text{when} \quad \mabfj =(j,\iota ) \ \text{and} \ \mabfk =(k,\kappa ).
\end{multline}

\par

By the Gabor analysis for modulation spaces we get the following. We
refer to \cite{Fei3,FeiGro1,FeiGro2,FeiGro3,GaSa,Gro1,Gro2,Toft13} for details.

\par

\begin{prop}\label{Prop:GaborExpMod}
Let $\sigma >1$, $s\ge 1$, $p,q\in (0,\infty ]$, $\omega \in \mascP _{E,s}^\sigma
(\rr {2d})$,
$\phi ,\psi \in \maclD ^\sigma (\rr d;[0,1])$ be such that
\eqref{Eq:phiProp}, \eqref{Eq:psiProp} and \eqref{Eq:PartUnity1}
hold true, and let $f\in (\maclD ^\sigma )'(\rr d)$. Then the following is true:
\begin{enumerate}
\item $f\in M^{p,q}_{(\omega )}(\rr d)$, if and only if
$\nm {V_\phi f}{\ell _{(\omega )}^{p,q}(\zz d\times \pi \zz d)}$;

\vrum

\item $f\in M^{p,q}_{(\omega )}(\rr d)$, if and only if
$\nm {V_\psi f}{\ell _{(\omega )}^{p,q}(\zz d\times \pi \zz d)}$;

\vrum

\item the quasi-norms
$$
f\mapsto \nm {V_\phi f}{\ell _{(\omega )}^{p,q}(\zz d\times \pi \zz d)}
\quad \text{and}\quad
f\mapsto \nm {V_\psi f}{\ell _{(\omega )}^{p,q}(\zz d\times \pi \zz d)}
$$
are equivalent to $\nm \cdo{M^{p,q}_{(\omega )}}$.
\end{enumerate}
The same holds true with $W^{p,q}_{(\omega )}$
and $\ell _{*,(\omega )}^{p,q}$ in place of
$M^{p,q}_{(\omega )}$
respectively $\ell _{(\omega )}^{p,q}$ at each occurrence.
\end{prop}

\par

\begin{rem}\label{Rem:GaborExpMod}
There are weights $\omega \in \mascP _E(\rr {2d})$ such that
corresponding modulation spaces
$M^{p,q}_{(\omega )}(\rr d)$ and $W^{p,q}_{(\omega )}(\rr d)$
do not contain $\maclD ^\sigma (\rr d)$ for any choice of $\sigma >1$.
In this situation, it is not possible to find compactly supported elements
in Gabor pairs which can be used for expanding all elements in
$M^{p,q}_{(\omega )}(\rr d)$ and $W^{p,q}_{(\omega )}(\rr d)$.

\par

For a general weight $\omega \in \mascP _E(\rr {2d})$ which is moderated by
the submultiplicative weight $v \in \mascP _E(\rr {2d})$, we may always find
a lattice $\Lambda \in \rr d$ and a Gabor pair $(\phi ,\psi )$ such that
$$
\phi ,\psi \in \bigcap _{p>0}M^p_{(v)}(\rr d)
$$
and
\begin{align}
f(x)
&=
C\sum _{j,\iota \in \Lambda}V_\phi f(j,\iota )\psi (x-j)e^{i\scal x\iota}
\notag
\\[1ex]
&
=
C\sum _{j,\iota \in \Lambda}V_\psi f(j,\iota )\phi (x-j)e^{i\scal x\iota},
\qquad f\in M^\infty _{(\omega )}(\rr d),
\label{Eq:GaborExpModGen}
\end{align}
for some constant $C$, where the series convergence with respect to the
weak$^*$ topology in
$M^\infty _{(\omega )}(\rr d)$. (See \cite[Theorem S]{Gro1} and some
further comments in \cite{Toft13}. See also \cite{FeiGro1,FeiGro2,FeiGro3}
for more facts.) In such approach we still have that if
$p,q\in (0,\infty ]$, then
\begin{alignat}{3}
f &\in M^{p,q} _{(\omega )}(\rr d) &
\ &\Leftrightarrow &\ 
\{ V_\phi f(j,\iota ) \} _{j,\iota \in \Lambda} &\in 
\ell ^{p,q} _{(\omega )}(\Lambda \times \Lambda )
\notag
\\
& &
\ &\Leftrightarrow &\ 
\{ V_\psi f(j,\iota ) \} _{j,\iota \in \Lambda} &\in 
\ell ^{p,q} _{(\omega )}(\Lambda \times \Lambda ),
\\[1ex]
f &\in W^{p,q} _{(\omega )}(\rr d) &
\ &\Leftrightarrow &\ 
\{ V_\phi f(j,\iota ) \} _{j,\iota \in \Lambda} &\in 
\ell ^{p,q} _{*,(\omega )}(\Lambda \times \Lambda )
\notag
\\
& &
\ &\Leftrightarrow &\ 
\{ V_\psi f(j,\iota ) \} _{j,\iota \in \Lambda} &\in 
\ell ^{p,q} _{*,(\omega )}(\Lambda \times \Lambda ),
\end{alignat}
\begin{alignat}{2}
\nm f{M^{p.q}_{(\omega )}}
& \asymp &
\nm {V_\phi f}{\ell ^{p.q}_{(\omega )}(\Lambda \times \Lambda)}
& \asymp
\nm {V_\psi f}{\ell ^{p.q}_{(\omega )}(\Lambda \times \Lambda)}
\intertext{and}
\nm f{W^{p.q}_{(\omega )}}
& \asymp &
\nm {V_\phi f}{\ell ^{p.q}_{*,(\omega )}(\Lambda \times \Lambda)}
& \asymp
\nm {V_\psi f}{\ell ^{p.q}_{*,(\omega )}(\Lambda \times \Lambda)}.
\end{alignat}
Furthermore, if $f\in M^{p.q}_{(\omega )}(\rr d)$
($f\in W^{p.q}_{(\omega )}(\rr d)$) and in addition $p,q<\infty$, then the series in
\eqref{Eq:GaborExpModGen} converges with respect to the $M^{p.q}_{(\omega )}$
quasi-norm ($W^{p.q}_{(\omega )}$ quasi-norm).
\end{rem}

\par

\begin{rem}\label{Rem:WienerSpaces}
Let $\omega _0\in \mascP _E(\rr d)$,
$\omega \in \mascP _E(\rr {2d})$, $p,q,r\in (0,\infty]$,
$Q_{d}=[0,1]^{d}$ be the unit cube,
and set for measurable $f$ on $\rr d$,
\begin{equation}\label{Eq:WienerQuasiNormsSimple}
\nm f{\sfW ^r(\omega ,\ell ^{p})} \equiv \nm {a_0}{\ell ^{p}(\zz {d})}
\end{equation}
when
$$
a_0(j) \equiv \nm {f\cdot \omega _0}{L^r (j+Q_{d})},
\qquad
j\in \zz d,
$$
and measurable $F$ on $\rr {2d}$,
\begin{equation}\label{Eq:WienerQuasiNorms}
\nm F{\sfW ^r(\omega ,\ell ^{p,q})} \equiv \nm a{\ell ^{p,q}(\zz {2d})}
\quad \text{and}\quad
\nm F{\sfW (\omega ,\ell ^{p,q}_*)} \equiv \nm a{\ell ^{p,q}_*(\zz {2d})}
\end{equation}
when
$$
a(j,\iota ) \equiv \nm {F\cdot \omega}{L^r ((j,\iota )+Q_{2d})},
\qquad
j,\iota \in \zz d.
$$
The Wiener space
$$
\sfW ^r(\omega _0,\ell ^{p}) = \sfW ^r(\omega _0,\ell ^{p}(\zz {2}))
$$
consists of all measurable $f\in L^r _{loc}(\rr {d})$
such that $\nm F{\sfW ^r(\omega _0,\ell ^{p})}$ is finite, and 
the Wiener spaces
$$
\sfW ^r(\omega ,\ell ^{p,q}) = \sfW ^r(\omega ,\ell ^{p,q}(\zz {2d}))
\quad \text{and}\quad
\sfW ^r(\omega ,\ell ^{p,q}_*) = \sfW ^r(\omega ,\ell ^{p,q}_*(\zz {2d}))
$$
consist of all measurable $F\in L^r _{loc}(\rr {2d})$
such that $\nm F{\sfW ^r(\omega ,\ell ^{p,q})}$ respectively
$\nm F{\sfW ^r(\omega ,\ell ^{p,q}_*)}$ are finite. The topologies
are defined through their respectively quasi-norms in
\eqref{Eq:WienerQuasiNormsSimple} and
\eqref{Eq:WienerQuasiNorms}.
%
%
%
%
%
%
For conveniency we set
$$
\sfW (\omega ,\ell ^{p,q})=\sfW ^\infty (\omega ,\ell ^{p,q})
\quad \text{and}\quad
\sfW (\omega ,\ell ^{p,q}_*)=\sfW ^\infty (\omega ,\ell ^{p,q}_*).
$$

\par

Obviously, $\sfW ^r(\omega _0,\ell ^{p})$ and
$\sfW ^r(\omega ,\ell ^{p,q})$ increase with $p,q$,
decrease with $r$, and
\begin{alignat}{1}
\sfW (\omega ,\ell ^{p,q})
&\hookrightarrow
L^{p,q}_{(\omega )}(\rr {2d}) \cap \Sigma _1'(\rr {2d})
\hookrightarrow
L^{p,q}_{(\omega )}(\rr {2d})
\hookrightarrow
\sfW ^r(\omega ,\ell ^{p,q})
\intertext{and}
\nm \cdo {\sfW ^r(\omega ,\ell ^{p,q})}
&\le
\nm \cdo {L^{p,q}_{(\omega )}}
\le
\nm \cdo {\sfW (\omega ,\ell ^{p,q})},
\qquad
r\le\min (1,p,q).
\end{alignat}
On the other hand, for modulation spaces we have
\begin{equation}\label{Eq:ModWienerConn1}
f\in M^{p,q}_{(\omega )}(\rr d)
\quad \Leftrightarrow \quad
V_\phi f\in L^{p,q}_{(\omega )}(\rr {2d})
\quad \Leftrightarrow \quad
V_\phi f\in \sfW ^r(\omega ,\ell ^{p,q})
\end{equation}
with
\begin{equation}\label{Eq:ModWienerConn2}
\nm f{M^{p,q}_{(\omega )}} = \nm {V_\phi f}{L^{p,q}_{(\omega )}}
\asymp
\nm {V_\phi f}{\sfW ^r(\omega ,\ell ^{p,q})}.
\end{equation}
The same holds true with $W^{p,q}_{(\omega )}$, $L^{p,q}_{*,(\omega )}$
and $\sfW (\omega ,\ell ^{p,q}_*)$ in place of $M^{p,q}_{(\omega )}$,
$L^{p,q}_{(\omega )}$ and $\sfW (\omega ,\ell ^{p,q})$, respectively,
at each occurrence. (For $r=\infty$ , see \cite{Gro2} when $p,q\in [1,\infty ]$,
\cite{GaSa,Toft13} when $p,q\in (0,\infty ]$, and for $r\in (0,\infty ]$, see
\cite{Toft25}.)
\end{rem}

\par

\subsection{Convolutions and multiplications for discrete Lebesgue spaces}

\par

Next we discuss extended H{\"o}lder and Young relations
for multiplications and convolutions on discrete Lebesgue spaces.
Here the involved weights should satisfy
\begin{align}
\omega _0(x) &\le \prod _{j=1}^N \omega _j(x)
\label{Eq:WeightHolder}
\intertext{or} 
\omega _0(x_1+\cdots +x_N) &\le \prod _{j=1}^N \omega _j(x_j),
\label{Eq:WeightYoung}
\end{align}
and it is convenient to make use of the functional
\begin{equation}\label{Eq:QFuncDef}
R_N(p_1,\dots ,p_N)
=
\left (\sum _{j=1}^N \max \left ( 1,\frac 1{p_j} \right ) \right )
-
\min _{1\le j\le N} \left (
\max \left ( 1,\frac 1{p_j} \right )
\right ).
\end{equation}
The
H{\"o}lder and Young conditions on Lebesgue exponent are then
\begin{align}
\frac 1{q_0} &\le \sum _{j=1}^N \frac 1{q_j},
\label{Eq:LebExpHolder}
\intertext{respectively}
\frac 1{p_0} &\le \sum _{j=1}^N \frac 1{p_j} - R_N(p_1,\dots ,p_N).
\label{Eq:LebExpYoung}
\end{align}

\par

\begin{prop}\label{Prop:HolderYoungDiscrLebSpaces}
Let $p_j, q_j\in (0,\infty ]$ be such that
\eqref{Eq:QFuncDef}, \eqref{Eq:LebExpHolder} and
\eqref{Eq:LebExpYoung} hold,
$\omega _j\in \mascP _E(\rr d)$, and let
$\Lambda \subseteq \rr d$ be a lattice containing origin.
Then the following is true:
\begin{enumerate}
\item if \eqref{Eq:WeightHolder} holds
true, then the map $(a_1,\dots ,a_N)\mapsto a_1\cdots a_N$ from
$\ell _0(\Lambda )\times \cdots \times \ell _0(\Lambda )$ to $\ell _0(\Lambda )$
extends uniquely to a continuous map from
$\ell ^{q_1}_{(\omega _1)}(\Lambda )\times \cdots
\times \ell ^{q_N}_{(\omega _N)}(\Lambda )$ to $\ell ^{q_0}_{(\omega _0)}(\Lambda )$,
and
\begin{equation}\label{Eq:HolderEst}
\nm {a_1\cdots a_N}{\ell ^{q_0}_{(\omega _0)}}
\le
\prod _{j=1^N}\nm {a_j}{\ell ^{q_j}_{(\omega _j)}},
\qquad a_j\in \ell ^{q_j}_{(\omega _j)}(\Lambda ),\ j=1,\dots ,N
\text ;
\end{equation}

\vrum

\item if \eqref{Eq:WeightYoung} 
holds
true, then the map $(a_1,\dots ,a_N)\mapsto a_1*\cdots *a_N$ from
$\ell _0(\Lambda )\times \cdots \times \ell _0(\Lambda )$ to $\ell _0(\Lambda )$
extends uniquely to a continuous map from
$\ell ^{p_1}_{(\omega _1)}(\Lambda )\times \cdots
\times \ell ^{p_N}_{(\omega _N)}(\Lambda )$ to $\ell ^{p_0}_{(\omega _0)}(\Lambda )$,
and
\begin{equation}\label{Eq:YoungEst}
\nm {a_1*\cdots *a_N}{\ell ^{p_0}_{(\omega _0)}}
\le
\prod _{j=1^N}\nm {a_j}{\ell ^{p_j}_{(\omega _j)}},
\qquad a_j\in \ell ^{p_j}_{(\omega _j)}(\Lambda ),\ j=1,\dots ,N
\text .
\end{equation}
\end{enumerate}
\end{prop}

\par

The assertion (1) in Proposition \ref{Prop:HolderYoungDiscrLebSpaces} is the standard
H{\"o}lder's inequality for discrete Lebesgue spaces. The assertion (2) in that
proposition is the usual Young's inequality for Lebesgue spaces on lattices
in the case when $p_1,\dots ,p_N\in [1,\infty ]$. In order to be self-contained we
give a proof when $p_1,\dots ,p_N$ are allowed to belong to the full interval
$(0,\infty ]$ in Appendix \ref{AppA}.

\par

\subsection{Step and Fourier step multipliers}

\par

Let $b\in \rr d_+$ be fixed, $\Lambda _b$ be the lattice given by
\begin{equation}\label{Eq:LatticebDef}
\Lambda _b =\sets {(b_1n_1,\dots ,b_dn_d)\in \rr d}{(n_1,\dots ,n_d)\in \zz d},
\end{equation}
$Q_b$ be the $b$-cube, given by
\begin{equation}\label{Eq:QubebDef}
Q_b = \sets {(b_1x_1,\dots ,b_dx_d)\in \rr d}{(x_1,\dots ,x_d)\in [0,1]^d}
\end{equation}
and $a_0\in \ell ^\infty (\Lambda _b)$. Then we let the Fourier step multiplier
$M_{\mascF \! ,b,a_0}$ (with respect to $b$ and $a_0$) be defined by
\begin{equation}\label{Eq:DefStepFourMult}
M_{\mascF \! ,b,a_0} \equiv \mascF ^{-1}\circ M_{b,a_0} \circ \mascF ,
\end{equation}
where $M_{b,a_0}$ is the multiplier
\begin{equation}\label{Eq:DefStepMult}
M_{b,a_0}\, :\, f\mapsto \sum _{j\in \Lambda _b}a_0(j)\chi _{j+Q_b}f.
\end{equation}
Here $\chi _\Omega$ is the characteristic function of $\Omega$.

\par

\section{Step and Fourier step multipliers on
modulation spaces}\label{sec2}

\par

In this section we deduce continuity properties for step and
Fourier step multipliers on modulation spaces (see Theorems
\ref{Thm:Mainthm1} and \ref{Thm:Mainthm2} below). In contrast to
\cite{BenGraGroOko}, the results presented here permit Lebesgue
exponents to be smaller than one

\par

We begin with step multipliers when acting on modulation spaces. Here
involved Lebesgue exponents should fullfil
\begin{equation}\label{Eq:MainThmCondLebExp}
\frac 1{q_1}-\frac 1{q_2} \ge \max \left ( \frac 1p-1,0 \right ), 
\end{equation}

\par

\begin{thm}\label{Thm:Mainthm1}
Let $p\in (0,\infty ]$, $q\in (1,\infty )$, $q_1,q_2\in (\min (1,p),\infty )$
be such that \eqref{Eq:MainThmCondLebExp} holds,
$b>0$,
$\omega _0\in \mascP _E(\rr {d})$ and $\omega (x,\xi )=\omega _0(x)$,
$x,\xi \in \rr d$. Let $a_0\in \ell ^\infty (\Lambda _b)$.
Then the following is true:
\begin{enumerate}
\item $M_{b,a_0}$ is continuous on $W^{p,q}_{(\omega )}(\rr d)$;

\vrum

\item  $M_{b,a_0}$ is continuous from $M^{p,q_1}_{(\omega )}(\rr d)$
to $M^{p,q_2}_{(\omega )}(\rr d)$.
\end{enumerate}
\end{thm}

\par

We observe that the conditions on $q_2$ in Theorem \ref{Thm:Mainthm1}
implies that $q_2>1$, since otherwise \eqref{Eq:MainThmCondLebExp}
should lead to $q_1\le \min (p,1)$, which contradicts the assumptions on
$q_1$.

\par

We need the following lemma for the proof of Theorem \ref{Thm:Mainthm1}.

\par

\begin{lemma}\label{Lemma:SingKernelConv}
Let $p,q\in (1,\infty )$ and $\theta \in (0,1)$ be such that
\begin{equation}\label{Eq:YoungCondInLemma}
\theta + \frac 1p=1+\frac 1q
\end{equation}
and suppose that $a=\{ a(j)\}_{j\in \zz d}\subseteq \mathbf C$ satisfies
$$
|a(j)| \lesssim \big ( \eabs {j_1}\cdots \eabs {j_d}\big )^{-\theta}.
$$
Then the map $b\mapsto a*b$ from $\ell _0(\zz d)$ to $\ell _0'(\zz d)$
is uniquely extendable to a continuous mapping from $\ell ^p(\zz d)$
to $\ell ^q(\zz d)$.
\end{lemma}

\par

We observe that the conditions in Lemma \ref{Lemma:SingKernelConv}
implies that $p<q$.

\par

Lemma \ref{Lemma:SingKernelConv} is a straight-forward consequence of
\cite[Theorem 4.5.3]{Ho1}. In fact, by that theorem we have for
\begin{equation}\label{Eq:Functionsh0h}
h(x)=(|x_1|\cdots |x_d|)^{-1}
\quad \text{and}\quad
h_0(x) = (\eabs {x_1}\cdots \eabs {x_d})^{-1}
\end{equation}
that
$$
\nm {f*h^{\theta}}{L^q(\rr d)}\lesssim \nm f{L^p(\rr d)},
$$
when \eqref{Eq:YoungCondInLemma} holds.
Since $0<h_0(x)<h(x)$, we obtain
\begin{equation}\label{Eq:SingOpContEst}
\nm {f*h_0^{\theta}}{L^q(\rr d)}
\le
\nm {|f|*h^{\theta}}{L^q(\rr d)}
\lesssim
\nm f{L^p(\rr d)},
\end{equation}
which gives suitable boundedness properties for $f\mapsto f*h_0^{\theta}$.

\par

We also have
$$
\nm {g_a}{L^p(\rr d)} = \nm a{\ell ^p(\zz d)},
\qquad
g_a(x) = \sum _{j\in \zz d} a(j)\chi _{j+[0,1]^d}(x).
$$
By a straight-forward combination of this estimate with
\eqref{Eq:SingOpContEst} we obtain
$$
\nm {b*h_0^{\theta}}{\ell ^q(\zz d)}
\lesssim
\nm b{\ell ^p(\zz d)},
$$
where now $*$ denotes the discrete convolution. The continuity
assertions in Lemma \ref{Lemma:SingKernelConv} now follows from
$$
\nm {b*a}{\ell ^q(\zz d)}
\lesssim
\nm {|b|*h_0^{\theta}}{\ell ^q(\zz d)}
\lesssim
\nm b{\ell ^p(\zz d)},
$$
and the uniqueness assertions follows from the fact that
$\ell _0(\zz d)$ is dense in $\ell ^p(\zz d)$ when $p<\infty$.

\par

\begin{proof}[Proof of Theorem \ref{Thm:Mainthm1}]
By straight-forward computations it follows that if
$\omega _b(x,\xi )= \omega (bx,b^{-1}\xi )$, then
$f\in W^{p,q}_{(\omega )}(\rr d)$, if and only if $f(b\cdo )\in W^{p,q}_{(\omega _b)}(\rr d)$,
and
$$
\nm f{W^{p,q}_{(\omega )}} \asymp \nm {f(b\cdo )}{W^{p,q}_{(\omega _b)}},
$$
and similarly with $M^{p,q}$ in place of $W^{p,q}$ at each occurrence.
This reduce ourself to the case when $b=1$.

\par

Let $\phi$, $\psi$ and $\Lambda$ be the same as in
\eqref{Eq:phiProp}--\eqref{Eq:PartUnity1} and \eqref{Eq:OurLattice}.
By \eqref{Eq:OpGaborMatrix} we have
\begin{equation}\label{Eq:MultAction2}
M_{b,a_0}f(x) = \left ( \frac \pi 2 \right )^{\frac d2} \sum _{(j,\iota )\in \Lambda}
(A\cdot V_\phi f))(j,\iota )e^{i\scal x\iota}\psi (x-j),
\end{equation}
where $A=(a(\mabfj ,K))_{\mabfj ,\mabfk \in \Lambda}$
is the matrix with elements
\begin{multline*}
a(\mabfj ,\mabfk )
=
\left ( \frac \pi 2\right )^{\frac d2}
(M_{b,a_0}e^{i\scal \cdo \iota}\psi(\cdo -j),
e^{i\scal \cdo \kappa}\phi (\cdo -k) ),
\\[1ex]
\text{when} \quad 
\mabfj  = (j,\iota )\in \Lambda ,\ \mabfk =(k,\kappa ) \in \Lambda .
\end{multline*}

\par

Let $Q=[0,1]^d$ and
$$
\Omega _m = \sets {j\in \zz d}{|j_n|\le m\ \text{for every}\ n\in \{ 1,\dots ,d\}} ,
\qquad
m\in \mathbf Z_+ .
$$
By the support properties of $\phi$ and $\psi$ we have
$$
a(\mabfj ,\mabfk ) = 0
\quad \text{when}\quad
j-k\notin \Omega _2 ,
$$
and for $j-k\in \Omega _2$ we get
\begin{multline}\label{Eq:MatrixCoeffEst1}
|a(\mabfj ,\mabfk )|
\asymp
\left |
\int _{\rr d}
M_{b,a_0}(e^{-i\scal \cdo  {\kappa -\iota}}\psi(\cdo -j)
)(y)\, 
\phi (y-k)\, dy
\right |
\\[1ex]
=
\left |
\int _{\rr d}
\left (
\sum _{l\in \zz d}a_0(l)\chi _Q(y-l)e^{-i\scal y {\kappa -\iota}}\psi(y-j)
\phi (y-k)\, dy
\right )
\right |
\\[1ex]
=
\left |
\int _{\rr d}
\left (
\sum a_0(l)\chi _Q(y-(l-k))e^{-i\scal y {\kappa -\iota}}\psi(y-(j-k))
\phi (y)\, dy
\right )
\right |
\\[1ex]
\le
\sum  |a_0(l)|\cdot
\left |
\int _{\rr d}
\left (
\chi _Q(y-(l-k))e^{-i\scal y {\kappa -\iota}}\psi(y-(j-k))
\phi (y)\, dy
\right )
\right |
\\[1ex]
\le
\nm {a_0}{\ell ^\infty (\zz d)}\sum 
\left |
(\chi _Q(\cdo -(l-k)),e^{-i\scal \cdo {\kappa -\iota}}\psi(\cdo-(j-k))
\phi )_{L^2(\rr d)}
\right | ,
\end{multline}
where the last three sums are taken over all $l\in \zz d$ such that
$l-(j-k)\in \Omega _3$.

\par

We have to estimate
$$
\left |
(\chi _Q(\cdo -(l-k)),e^{-i\scal \cdo {\kappa -\iota}}\psi(\cdo-(j-k))
\phi )_{L^2(\rr d)}
\right |
$$
when $j-k\in \Omega _2$ and $l-(j-k)\in \Omega _3$. By Parseval's formula
we get
\begin{multline*}
|(\chi _Q(\cdo -(l-k)),e^{-i\scal \cdo {\kappa -\iota}}\psi(\cdo-(j-k))
\phi )_{L^2(\rr d)}|
\\[1ex]
=
|(e^{-i\scal {l-k}\cdo}g,V_\psi \phi (j-k,\cdo -(\iota -\kappa )))_{L^2(\rr d)}|,
\end{multline*}
where
$$
g(\xi ) = (2\pi )^{-\frac d2} e^{-\frac i2(\xi _1+\cdots +\xi _d)}
\sinc (\xi _1/2)\cdots \sinc (\xi _d/2).
$$
Here
$$
\sinc t
= 
\begin{cases}
\frac {\sin t}t, & t\neq 0,
\\[1ex]
1, & t=0,
\end{cases}
$$
is the sinc function.

\par

Since
$$
\sinc t\lesssim \eabs t^{-1} \quad \text{and}\quad
|V_\psi \phi (j-k,\xi -(\iota -\kappa ))|\lesssim e^{-r|\xi -(\iota -\kappa )|^{\frac 1s}}
$$
we obtain
\begin{multline*}
|(\chi _Q(\cdo -(l-k)),e^{-i\scal \cdo {\kappa -\iota}}\psi(\cdo-(j-k))
\phi )_{L^2(\rr d)}|
\\[1ex]
\le
\int _{\rr d} |g(\xi )|\cdot |V_\psi \phi (j-k,\xi -(\iota -\kappa ))|\, d\xi
\\[1ex]
\lesssim
\int _{\rr d} h_0(\xi )
e^{-r|\xi -(\iota -\kappa )|^{\frac 1s}}\, d\xi
\\[1ex]
=
\int _{\rr d}
h_0(\xi +\iota -\kappa )
e^{-r|\xi |^{\frac 1s}}\, d\xi
\\[1ex]
\le
h_0(\iota -\kappa )
\int _{\rr d}
\eabs {\xi _1}\cdots \eabs {\xi _d}
e^{-r|\xi |^{\frac 1s}}\, d\xi
\asymp
h_0(\iota -\kappa ).
\end{multline*}
Here $h_0$ is given by \eqref{Eq:Functionsh0h}, and we have used
$$
h_0(\xi +\eta) = (\eabs {\xi _1+\eta _1}\cdots \eabs{\xi _d+\eta _d})^{-1}
\le
(\eabs {\eta _1}\cdots \eabs{\eta _d})^{-1}\eabs {\xi _1}\cdots \eabs{\xi _d}.
$$

\par

By inserting this into \eqref{Eq:MatrixCoeffEst1} we get
\begin{equation*}
|a(\mabfj ,\mabfk )| \lesssim \sum_{l\in (j-k) +\Omega _3}
h_0(\iota -\kappa )
(\eabs {\iota _1-\kappa _1}\cdots \eabs {\iota _d-\kappa _d})^{-1}
=
7^d h_0(\iota -\kappa ).
\end{equation*}
Hence,
\begin{equation}\label{Eq:FinalMatrixCoeffEst}
|a(\mabfj ,\mabfk )|
\lesssim
\begin{cases}
h_0(\iota -\kappa ),
& j-k\in \Omega _2
\\[1ex]
0, & j-k\notin \Omega _2 .
\end{cases}
\end{equation}

\par

If $c(j,\iota )=|V_\phi f(j,\iota )|$, then
\eqref{Eq:FinalMatrixCoeffEst} gives
\begin{multline}\label{Eq:AOpEst}
|(A\cdot V_\phi f)(j ,\iota )| \lesssim \sum _{k\in j+\Omega _2}
\left (
\sum _{\kappa \in \pi \zz d} h_0(\iota -\kappa )c(k,\kappa)
\right )
\\[1ex]
=
\sum _{k\in j+\Omega _2} (h_0*c(k,\cdo ))(\iota )
=
\sum _{k\in \Omega _2} (h_0*c(j+k,\cdo ))(\iota ),
\end{multline}
and Lemma \ref{Lemma:SingKernelConv} gives
\begin{multline*}
\nm {(A\cdot V_\phi f)(j ,\cdo )}{\ell ^q(\pi \zz d)}
\lesssim 
\sum _{k\in \Omega _2} \nm {h_0*c(j+k,\cdo )}{\ell ^q(\pi \zz d)}
\\[1ex]
\lesssim
\sum _{k\in \Omega _2} \nm {c(j+k,\cdo )}{\ell ^q(\pi \zz d)}.
\end{multline*}

\par

By applying the $\ell ^p_{(\omega _0)}$ norm on the last inequality and raise it
to the power $r=\min (1,p)$, we obtain
\begin{multline}\label{Eq:DiscLebEstMatrixOp}
\nm {(A\cdot V_\phi f)}{\ell ^{p,q}_{*,(\omega )}(\Lambda )}^r
\lesssim
\sum _{k\in \Omega _2} \nm {c(\cdo +(k,0))}{\ell _{*,(\omega )}^{p,q}(\Lambda )}^r
\\[1ex]
=
\sum _{k\in \Omega _2} \nm {c}{\ell _{*,(\omega (\cdo -(k,0)))}^{p,q}(\Lambda )}^r
\asymp
\sum _{k\in \Omega _2} \nm {c}{\ell _{*,(\omega )}^{p,q}(\Lambda )}^r
=
5^d \nm {c}{\ell _{*,(\omega )}^{p,q}(\Lambda )}^r.
\end{multline}
Here we have used the fact that the number of elements in $\Omega _2$ is equal to
$5^d$.

\par

The asserted continuity in (1) now follows in the case when $p<\infty$ by combining
\eqref{Eq:DiscLebEstMatrixOp} and the facts that
$$
\nm {V_\phi f}{\ell ^{p,q}_{*,(\omega )}(\Lambda )}
\asymp
\nm {f}{W^{p,q}_{(\omega)}}
\quad \text{and}\quad
\nm {(A\cdot V_\phi f)}{\ell ^{p,q}_{*,(\omega )}(\Lambda )}
\asymp
\nm {M_{b,a_0}f}{W^{p,q}_{(\omega)}}.
$$
The uniqueness of the map $M_{b,a_0}$ on $W^{p,q}_{(\omega )}(\rr d)$
follows from the fact that finite sequences in \eqref{Eq:SpecGaborExp2}
are dense in $W^{p,q}_{(\omega )}(\rr d)$ gives. The case when $p=\infty$
now follows from the case when $p=1$ and duality, and (1) follows.

\par

In order to prove (2) we first consider the case when
$p<\infty$. By applying the $\ell ^p_{(\omega _0)}$
norm with respect to the $j$ variable in \eqref{Eq:AOpEst}, we get
\begin{multline*}
\nm {(A\cdot V_\phi f)(\cdo ,\iota )}{\ell _{(\omega _0)}^p}
\lesssim
\left (
\sum _{j\in \zz d}
\left ( 
\sum _{k\in \Omega _2}
(h_0*  (c(j+k,\cdo ))(\iota )\omega _0(j)
\right )^p \right )^{\frac 1p}
\\[1ex]
\asymp
\sum _{k\in \Omega _2}
\left (
h_0^r* \left ( \sum _{j\in \zz d} (c(j+k,\cdo )\omega _0(j+k))^p \right )^{\frac rp}
\right ) (\iota )
\\[1ex]
= 
\sum _{k\in \Omega _2}
( h_0^r*c_0^r)(\iota )
\asymp
( h_0^r*c_0^r)(\iota ),
\end{multline*}
where $c_0(\iota) = \nm {c(\cdo ,\iota)}{\ell ^p_{(\omega _0)}}$.

\par

Let $p_0=r^{-1}q_1$, $q_0=r^{-1}q_2$ and $u=r^{-1}$. Then
\eqref{Eq:YoungCondInLemma} holds with $p_0$ and $q_0$ in place
of $p$ and $q$, respectively. Hence by applying the $\ell ^{q_0}$ norm
on the last estimates, Lemma \ref{Lemma:SingKernelConv} gives
\begin{equation*}
\nm {A\cdot V_\phi f}{\ell _{(\omega )}^{p,q_2}}^r
\lesssim
\nm {h_0^{\theta}*c_0^r}{\ell ^{q_0}}
\lesssim
\nm {c_0^r}{\ell ^{p_0}}
= 
\nm {c_0}{\ell ^{q_1}}^r.
\end{equation*}

\par

The asserted continuity in (2) now follows in the case when $p<\infty$ by combining
\eqref{Eq:DiscLebEstMatrixOp} and the facts that
$$
\nm {V_\phi f}{\ell ^{p,q_1}_{(\omega )}(\Lambda )}
\asymp
\nm {f}{M^{p,q_1}_{(\omega)}}
\quad \text{and}\quad
\nm {(A\cdot V_\phi f)}{\ell ^{p,q_2}_{(\omega )}(\Lambda )}
\asymp
\nm {M_{b,a_0}f}{M^{p,q_2}_{(\omega)}}.
$$
The uniqueness assertions as well as the continuity in the case $p=\infty$
follow by similar arguments as in the proof of (1). The details are left for the
reader.
\end{proof}

\par

By the links between $M^{p,q}_{(\omega )}(\rr d)$ and
$W^{p,q}_{(\omega )}(\rr d)$ via the Fourier transform, explained in
Remark \ref{Rem:ModSpaces}, the following result follows from
Theorem \ref{Thm:Mainthm1} and Fourier transformation. The details
are left for the reader.

\par

\begin{thm}\label{Thm:Mainthm2}
Let $p\in (1,\infty )$, $q\in (0,\infty ]$, $p_1,p_2\in (\min (1,q),\infty )$
be such that
\begin{equation}\label{Eq:MainThmCondLebExp2}
\frac 1{p_1}-\frac 1{p_2} \ge \max \left ( \frac 1q-1 ,0 \right ), 
\end{equation}
$b>0$, $a_0\in \ell ^\infty (\Lambda _b)$,
$\omega _0\in \mascP _{E,s}(\rr {d})$ and $\omega (x,\xi )=\omega _0(\xi )$,
$x,\xi \in \rr d$. Then the following is true:
\begin{enumerate}
\item $M_{\mascF \! ,b,a_0}$ is continuous on $M^{p,q}_{(\omega )}(\rr d)$;

\vrum

\item  $M_{\mascF \! ,b,a_0}$ is continuous from $W^{p_1,q}_{(\omega )}(\rr d)$
to $W^{p_2,q}_{(\omega )}(\rr d)$.
\end{enumerate}
\end{thm}

\par

We observe that Theorem \ref{Thm:Mainthm2}
generalizes \cite[Theorem 1]{BenGraGroOko} and \cite[Theorem 4.16]{Zan}.

\par

\section{Multiplications and convolutions of quasi-Banach
modulation spaces}\label{sec3}

\par

In this section we extend the multiplication and convolution properties
on modulation spaces in \cite{Fei1,Toft3} to allow the Lebesgue
exponents to belong to the full
interval $(0,\infty ]$ instead of $[1,\infty ]$, and to allow general
moderate weights. There are several approaches
in the case when the involved Lebesgue exponents belong to
$[1,\infty ]$ (see \cite{CorGro,Fei1,FeiGro1,GuChFaZh,RSTT,Toft3}). There are
also some results when such exponents belong to the full interval
$(0,\infty ]$ (see \cite{BaCoNi,BaTe,GaSa,Rau1,Rau2,Toft13}). Here we remark that
our results in this section cover several of these earlier results. For example,
we observe that Theorem \ref{Thm:MultMod1} below
extends \cite[Proposition 3.1]{BaCoNi}.

\par

%
We recall that convolutions and multiplications on $\Sigma _1(\rr d)$
are commutative and associative. 
That is, for any $N\ge 1$, $f_1,\dots ,f_N\in \Sigma _1(\rr d)$ and
$j,k\in \{ 1,\dots ,N\}$ one has
\begin{align*}
f_1\cdots f_N = (f_1\cdots f_j)\cdot (f_{j+1}\cdots f_N)
\quad \text{and}\quad
f_1\cdots f_N = g_1\cdots g_N
\end{align*}
when
$$
g_m=f_m, \quad g_j=f_k
\quad \text{and}\quad g_k=f_j,\ m\neq j,k,
$$
and similarly for convolutions in place of multiplications at each occurrence.

\par

Because of possible lacks of density properties, we do not always
reach the uniqueness when
extending the convolutions and multiplications from the case when
each $f_j$ belong to $\Sigma _1(\rr d)$ to the case when
each $f_j$ belong to suitable modulation spaces.
In some cases we manage the
uniqueness by replacing the (quasi-)norm convergence by
a weaker convergence, the so-called narrow convergence (see 
\cite{Sjo,Toft2,Toft10}).
In the other situations we define multiplications and convolutions in
terms of short-time Fourier transforms, in similar ways as in \cite{Toft3}.

\medspace

Let $\phi _0,\dots ,\phi _N\in \Sigma _1(\rr d)$ be fixed
such that
\begin{equation}\label{Eq:MultDefFormula2A}
(\phi _1\cdots \phi _N,\phi _0)_{L^2}=(2\pi )^{-(N-1)\frac {d}2}
\end{equation}
and let $f_1,\dots ,f_N,g\in \Sigma _1(\rr d)$. Then the multiplication
$f_1\cdots f_N$ can be expressed by
\begin{multline}\label{Eq:MultDefFormula1}
(f_1  \cdots  f_N,\fy ) _{L^2(\rr d)}
= \iint _{\rr d\times \rr {Nd}} 
\left (
\prod _{j=1}^NF_j(x,\xi _j)
\right )
\overline {\Phi (x,\xi _1+\cdots +\xi _N)}\, dxd\xi 
\\[1ex]
= \idotsint _{\rr{(N+1)d}} 
\left (
\prod _{j=1}^NF_j(x,\xi _j)
\right )
\overline {\Phi (x,\xi _1+\cdots +\xi _N)}\, dxd\xi _1\cdots d\xi _N
\end{multline}
for every $\fy \in \Sigma _1(\rr d)$, where
\begin{equation}
F_j = V_{\phi _j}f_j
\quad \text{and}\quad
\Phi =V_{\phi _0}\fy .
\label{Eq:MultDefFormula2}
\end{equation}
We observe that \eqref{Eq:MultDefFormula1} is the same as
\begin{align}
F_0(x,\xi )
&=
\big ( (V_{\phi _1}f_1)(x,\cdo ) *\cdots * (V_{\phi _N}f_N)(x,\cdo ) \big ) (\xi ).
\tag*{(\ref{Eq:MultDefFormula1})$'$}
\intertext{where}
F_0(x,\xi )
&=
(\nm {\phi _0}{L^2})^{-2}\cdot V_{\phi _0}(f_1\cdots f_N)(x,\xi ),
\label{Eq:F0ProdDef}
\end{align}
and that we may extract $f_0=f_1\cdots f_N$ by the formula
\begin{equation}\label{Eq:f0Extract}
f_0 = V_{\phi _0}^*F_0 .
\end{equation}

\par

In the same way, let $\phi _0,\dots ,\phi _N\in \Sigma _1(\rr d)$ be fixed such that
\begin{equation}\label{Eq:ConvDefFormula2A}
(\phi _1*\cdots *\phi _N,\phi _0)_{L^2}=1
\end{equation}
and let $f_1,\dots ,f_N,g\in \Sigma _1(\rr d)$. Then the convolution $f_1*\cdots *f_N$
can be expressed by
\begin{multline}\label{Eq:ConvDefFormula1}
(f_1*  \cdots  *f_N,\fy ) _{L^2(\rr d)}
= \iint _{\rr {Nd} \times \rr d} 
\left (
\prod _{j=1}^NF_j(x_j,\xi )
\right )
\overline {\Phi (x_1+\cdots +x_N,\xi )}\, dx d\xi 
\\[1ex]
= \idotsint _{\rr{(N+1)d}} 
\left (
\prod _{j=1}^NF_j(x_j,\xi )
\right )
\overline {\Phi (x_1+\cdots +x_N,\xi )}\, dx_1\cdots dx_Nd\xi ,
\end{multline}
for every $\fy \in \Sigma _1(\rr d)$, where $F_j$ and $\Phi$ are
given by \eqref{Eq:MultDefFormula2}.
We observe that \eqref{Eq:ConvDefFormula1} is the same as
\begin{align}
F_0(x,\xi )
&=
\big ( (V_{\phi _1}f_1)(\cdo ,\xi ) *\cdots * (V_{\phi _N}f_N)(\cdo ,\xi ) \big ) (x).
\tag*{(\ref{Eq:ConvDefFormula1})$'$}
\intertext{where}
F_0
&=
(\nm {\phi _0}{L^2})^{-2}V_{\phi _0}(f_1*\cdots *f_N),
\label{Eq:F0ConvDef}
\end{align}
and that we may extract $f_0=f_1*\cdots *f_N$ from \eqref{Eq:f0Extract}.

\par

\begin{defn}\label{Def:MultConvSTFT}
Let $f_1,\dots ,f_N\in \Sigma _1'(\rr d)$.
\begin{enumerate}
\item Let $\phi _0,\dots ,\phi _N\in \Sigma _1(\rr d)$ be fixed
and such that \eqref{Eq:MultDefFormula2A} holds, and suppose
that the integrand in \eqref{Eq:MultDefFormula1} belongs to
$L^1(\rr {(N+1)d})$ for every $\fy \in \Sigma _1(\rr d)$, where
$F_j=V_{\phi _j}f_j$ and $\Phi =V_{\phi _0}\fy$, $j=1,\dots ,N$.
Then $f_0\equiv f_1\cdots f_N\in \Sigma _1'(\rr d)$ is defined by
\eqref{Eq:MultDefFormula1};

\vrum

\item Let $\phi _0,\dots ,\phi _N\in \Sigma _1(\rr d)$ be fixed
and such that \eqref{Eq:ConvDefFormula2A} holds, and suppose
that the integrand in \eqref{Eq:ConvDefFormula1} belongs to
$L^1(\rr {(N+1)d})$ for every $\fy \in \Sigma _1(\rr d)$, where
$F_j=V_{\phi _j}f_j$ and $\Phi =V_{\phi _0}\fy$, $j=1,\dots ,N$.
Then $f_0\equiv f_1*\cdots *f_N\in \Sigma _1'(\rr d)$ is defined by
\eqref{Eq:ConvDefFormula1}.
\end{enumerate}
\end{defn}

\par

Next we discuss convolutions and multiplications for modulation spaces,
and start with the following convolution result for modulation spaces.
Here the conditions for the involved weight functions are given by
\begin{alignat}{2}
\omega _0(x,\xi _1+\cdots +\xi _N)
\lesssim \prod _{j=1}^N \omega _j(x,\xi _j),
\qquad
x,\xi _1,\dots ,\xi _N \in \rr d
\label{Eq:WeightModMult}
\intertext{or by}
\omega _0(x_1+\dots +x_N,\xi )
\lesssim \prod _{j=1}^N \omega _j(x_j,\xi ),
\qquad
x_1,\dots ,x_N,\xi \in \rr d.
\label{Eq:WeightModConv}
\end{alignat}
For multiplications of elements in modulation spaces we need
to swap the conditions for the involved Lebesgue exponents
compared to \eqref{Eq:LebExpHolder}
and \eqref{Eq:LebExpYoung}. That is, these conditions become
\begin{alignat}{2}
\frac 1{p_0} &\le \sum _{j=1}^N \frac 1{p_j},&
\qquad
\frac 1{q_0} &\le \sum _{j=1}^N \frac 1{q_j} - R_{p_0,N}(q_1,\dots ,q_N)
\label{Eq:LebExpHolderYoung1}
\intertext{or}
\frac 1{p_0} &\le \sum _{j=1}^N \frac 1{p_j},&
\qquad
\frac 1{q_0} &\le \sum _{j=1}^N \frac 1{q_j} - R_N(q_1,\dots ,q_N),
\label{Eq:LebExpHolderYoung2}
\end{alignat}
where
\begin{align}
R_{r,N}(q_1,\dots ,q_N)
&=
\left (\sum _{j=1}^N \frac 1{r_j} \right )
-
\min _{1\le j\le N} \left (
\frac 1{r_j}
\right ),\quad
r_j = \min (1,q_j,r)
\label{Eq:QrFuncDef}
\\[1ex]
\intertext{and}
R_N(q_1,\dots ,q_N) &= R_{1,N}(q_1,\dots ,q_N).
\label{Eq:QrFuncSpecDef}
\end{align}
Evidently, $R_{r,N}(q_1,\dots ,q_N)=R_N(q_1,\dots ,q_N)$ when
$r\ge 1$.

\par

\begin{thm}\label{Thm:MultMod1}
Let $I_N=\{ 1,\dots ,N\}$, $\omega _j\in \mascP _E(\rr {2d})$
and $p_j, q_j\in (0,\infty ]$, $j\in I_N$,
be such that \eqref{Eq:WeightModMult},
\eqref{Eq:LebExpHolderYoung1} and
\eqref{Eq:QrFuncDef} hold. Then
$(f_1,\dots ,f_N)\mapsto f_1\cdots f_N$ in
Definition \ref{Def:MultConvSTFT} (1) 
restricts to a continuous, associative and symmetric
map from $M^{p_1,q_1}_{(\omega _1)}(\rr d)
\times \cdots \times M^{p_N,q_N}_{(\omega _N)}(\rr d)$
to $M^{p_0,q_0}_{(\omega _0)}(\rr d)$, and
\begin{equation}\label{Eq:MultMod1}
\nm {f_1\cdots f_N}{M^{p_0,q_0}_{(\omega _0)}}
\lesssim
\prod _{j=1}^N \nm {f_j}{M^{p_j,q_j}_{(\omega _j)}} ,
\quad
f_j\in M^{p_j,q_j}_{(\omega _j)}(\rr d),\ j\in I_N \text .
\end{equation}
Moreover, $f_1\cdots f_N$ in \eqref{Eq:MultDefFormula1} is independent
of the choice of $\phi _0,\dots ,\phi _N$ in Definition
\ref{Def:MultConvSTFT} (1).
\end{thm}

\par

\begin{thm}\label{Thm:MultMod2}
Let $I_N=\{ 1,\dots ,N\}$, $\omega _j\in \mascP _E(\rr {2d})$
and $p_j, q_j\in (0,\infty ]$, $j\in I_N$,
be such that \eqref{Eq:WeightModMult},
\eqref{Eq:LebExpHolderYoung2} and
\eqref{Eq:QrFuncSpecDef} hold. Then
$(f_1,\dots ,f_N)\mapsto f_1\cdots f_N$ in
Definition \ref{Def:MultConvSTFT} (1) 
restricts to  a continuous, associative and symmetric
map from $W^{p_1,q_1}_{(\omega _1)}(\rr d)
\times \cdots \times W^{p_N,q_N}_{(\omega _N)}(\rr d)$
to $W^{p_0,q_0}_{(\omega _0)}(\rr d)$, and
\begin{equation}\label{Eq:MultMod2}
\nm {f_1\cdots f_N}{W^{p_0,q_0}_{(\omega _0)}}
\lesssim
\prod _{j=1}^N \nm {f_j}{W^{p_j,q_j}_{(\omega _j)}} ,
\quad
f_j\in W^{p_j,q_j}_{(\omega _j)}(\rr d),\ j\in I_N \text .
\end{equation}
Moreover, $f_1\cdots f_N$ in \eqref{Eq:MultDefFormula1} is independent
of the choice of $\phi _0,\dots ,\phi _N$ in Definition
\ref{Def:MultConvSTFT} (1).
\end{thm}

\par

The corresponding results for convolutions are the following. Here
the conditions on the involved Lebesgue exponents are swapped as
\begin{alignat}{2}
\frac 1{p_0} &\le \sum _{j=1}^N \frac 1{p_j} - R_{q_0,N}(p_1,\dots ,p_N), &
\qquad
\frac 1{q_0} &\le \sum _{j=1}^N \frac 1{q_j} 
\label{Eq:LebExpYoungHolder1}
\intertext{or}
\frac 1{p_0} &\le \sum _{j=1}^N \frac 1{p_j} - R_N(p_1,\dots ,p_N), &
\qquad
\frac 1{q_0} &\le \sum _{j=1}^N \frac 1{q_j} ,
\label{Eq:LebExpYoungHolder2}
\end{alignat}

\par

\begin{thm}\label{Thm:ConvMod1}
Let $I_N=\{ 1,\dots ,N\}$, $\omega _j\in \mascP _E(\rr {2d})$
and $p_j, q_j\in (0,\infty ]$, $j\in I_N$,
be such that \eqref{Eq:WeightModConv}, \eqref{Eq:QrFuncSpecDef}
and \eqref{Eq:LebExpYoungHolder2} hold. Then
$(f_1,\dots ,f_N)\mapsto f_1*\cdots *f_N$ in
Definition \ref{Def:MultConvSTFT} (2) 
restricts to a continuous, associative and symmetric
map from $M^{p_1,q_1}_{(\omega _1)}(\rr d)
\times \cdots \times M^{p_N,q_N}_{(\omega _N)}(\rr d)$
to $M^{p_0,q_0}_{(\omega _0)}(\rr d)$, and
\begin{equation}\label{Eq:ConvMod1}
\nm {f_1*\cdots *f_N}{M^{p_0,q_0}_{(\omega _0)}}
\lesssim
\prod _{j=1}^N \nm {f_j}{M^{p_j,q_j}_{(\omega _j)}} ,
\quad
f_j\in M^{p_j,q_j}_{(\omega _j)}(\rr d),\ j\in I_N \text .
\end{equation}
Moreover, $f_1*\cdots *f_N$ in \eqref{Eq:ConvDefFormula1}
is independent of the choice of $\phi _0,\dots ,\phi _N$ in Definition
\ref{Def:MultConvSTFT} (2).
\end{thm}

\par

\begin{thm}\label{Thm:ConvMod2}
Let $I_N=\{ 1,\dots ,N\}$, $\omega _j\in \mascP _E(\rr {2d})$
and $p_j, q_j\in (0,\infty ]$, $j\in I_N$,
be such that \eqref{Eq:WeightModConv}, \eqref{Eq:QrFuncDef}
and \eqref{Eq:LebExpYoungHolder1} hold. Then
$(f_1,\dots ,f_N)\mapsto f_1*\cdots *f_N$ in
Definition \ref{Def:MultConvSTFT} (2) 
restricts to a continuous, associative and symmetric
map from $W^{p_1,q_1}_{(\omega _1)}(\rr d)
\times \cdots \times W^{p_N,q_N}_{(\omega _N)}(\rr d)$
to $W^{p_0,q_0}_{(\omega _0)}(\rr d)$, and
\begin{equation}\label{Eq:ConvMod2}
\nm {f_1*\cdots *f_N}{W^{p_0,q_0}_{(\omega _0)}}
\lesssim
\prod _{j=1}^N \nm {f_j}{W^{p_j,q_j}_{(\omega _j)}} ,
\quad
f_j\in W^{p_j,q_j}_{(\omega _j)}(\rr d),\ j\in I_N \text .
\end{equation}
Moreover, $f_1*\cdots *f_N$ in \eqref{Eq:ConvDefFormula1}
is independent of the choice of $\phi _0,\dots ,\phi _N$ in Definition
\ref{Def:MultConvSTFT} (2).
\end{thm}

\par

For the proofs of Theorems \ref{Thm:MultMod1}--\ref{Thm:ConvMod2}
we need the following proposition. Here recall 
\cite{FeiGro1,FeiGro3,Gro2,Rau1,Rau2}
and Remark \ref{Rem:BasicPropTwistConv} for some facts concerning the 
operators
$P_\phi$ and $V_\phi ^*$.

\par

\begin{prop}\label{Prop:ProjMapModCont}
Let $p,q\in (0,\infty ]$, $\omega \in \mascP _E(\rr {2d})$,
$\phi \in \Sigma _1(\rr d)\setminus 0$ and $P_\phi$ be the projection
in Remark \ref{Rem:BasicPropTwistConv}. Then
$P_\phi$ from $\Sigma _1'(\rr {2d})$ to $\Sigma _1'(\rr {2d})$,
and $V_\phi ^*$ from $\Sigma _1'(\rr {2d})$ to $\Sigma _1'(\rr d)$
restrict to continuous mappings
\begin{align}
P_\phi : \sfW (\omega ,\ell ^{p,q}(\zz {2d}))
&\to
V_\phi (M^{p,q}_{(\omega )}(\rr d))
\hookrightarrow
\sfW (\omega ,\ell ^{p,q}(\zz {2d})),
\label{Eq:ProjMapModCont1}
\\[1ex]
P_\phi : \sfW (\omega ,\ell ^{p,q}_*(\zz {2d}))
&\to 
V_\phi (W^{p,q}_{(\omega )}(\rr d))
\hookrightarrow
\sfW (\omega ,\ell ^{p,q}_*(\zz {2d})),
\label{Eq:ProjMapModCont2}
\\[1ex]
V_\phi ^* : \sfW (\omega ,\ell ^{p,q} (\zz {2d}))
&\to 
M^{p,q}_{(\omega )}(\rr d)
\label{Eq:ProjMapModCont3}
\intertext{and}
V_\phi ^* : \sfW (\omega ,\ell ^{p,q}_*(\zz {2d}))
&\to 
W^{p,q}_{(\omega )}(\rr d).
\label{Eq:ProjMapModCont4}
\end{align}
\end{prop}

\par

For $p,q\ge 1$, i.{\,}.e. the case when all spaces are Banach spaces,
proofs of Proposition \ref{Prop:ProjMapModCont} can be found
in e.{\,}g. \cite{Gro2} as well as in abstract forms in \cite{FeiGro1}.
In the general case when $p,q>0$, proofs of
Proposition \ref{Prop:ProjMapModCont} are essentially given in
\cite{GaSa,Rau2}. In order to be self-contained we here present
a short proof.

\par

\begin{proof}
By Remark \ref{Rem:BasicPropTwistConv}, the result follows if we prove
\eqref{Eq:ProjMapModCont1} and \eqref{Eq:ProjMapModCont2}, i.{\,}e.,
it suffices to prove
\begin{alignat}{2}
\nm {P_\phi F}{\sfW (\omega ,\ell ^{p,q})} &\lesssim \nm {F}{\sfW (\omega ,\ell ^{p,q})}, &
\qquad F&\in \sfW (\omega ,\ell ^{p,q}(\zz {2d}))
\label{Eq:ProjMapContForm1}
\intertext{and}
\nm {P_\phi F}{\sfW (\omega ,\ell ^{p,q}_*)} &\lesssim
\nm {F}{\sfW (\omega ,\ell ^{p,q}_*)}, &
\qquad F&\in {\sfW (\omega ,\ell ^{p,q}_*(\zz {2d}))}.
\label{Eq:ProjMapContForm2}
\end{alignat}
We only prove \eqref{Eq:ProjMapContForm1}.
The estimate \eqref{Eq:ProjMapContForm2} follows by similar arguments and is left
for the reader.

\par

Let
$$
a_j = \nm F{L^\infty (j+Q_{2d})}
\quad \text{and}\quad
b_j = \nm {V_\phi \phi}{L^\infty (j+Q_{2d})}.
$$
Since $V_\phi \phi \in \Sigma _1(\rr {2d})$,
Proposition \ref{Prop:HolderYoungDiscrLebSpaces} gives
\begin{equation*}
\nm {P_\phi F}{\sfW (\omega ,\ell ^{p,q})}
\lesssim
\nm {a*b}{\ell ^{p,q}}
\lesssim
\nm b{\ell ^{\min (1,p,q)}}
\nm a{\ell ^{p,q}}
\asymp
\nm F{\sfW (\omega ,\ell ^{p,q})}. \qedhere
\end{equation*}
\end{proof}

\par

Theorems \ref{Thm:MultMod1} and \ref{Thm:MultMod2} are Fourier
transformations of Theorems \ref{Thm:ConvMod1} and \ref{Thm:ConvMod2}.
Hence it suffices to prove the last two theorems.

\par

\begin{proof}[Proof of Theorems \ref{Thm:ConvMod1} and \ref{Thm:ConvMod2}]
First we prove \eqref{Eq:ConvMod1}.
Suppose
$f_j\in M^{p_j,q_j}_{(\omega _j)}(\rr d)$, and consider the cubes
$$
Q_{d,r} =[0,r]^d
\quad \text{and}\quad
Q= Q_{d,1}= [0,1]^d.
$$ 
Then
$$
0\le \chi _{k_1+Q}*\cdots *\chi _{k_N+Q} \le \chi _{k_1+\cdots +k_N+Q_{d,N}},
\qquad
k_1,\dots k_N\in \zz d.
$$

\par

Let
\begin{align*}
G_1(x,\xi ) &= (V_{\phi _1}f_1(\cdo ,\xi )*\cdots *V_{\phi _N}f_N(\cdo ,\xi ))(x),
\\[1ex]
G_2(x,\xi ) &= (|V_{\phi _1}f_1(\cdo ,\xi )|*\cdots *|V_{\phi _N}f_N(\cdo ,\xi )|)(x),
\\[1ex]
a_j(k,\kappa ) &= \nm {V_{\phi _j}f_j}{L^\infty ((k,\kappa )+Q_{2d,1})},
\intertext{and}
b(k,\kappa ) &=  \nm {G_2}{L^\infty ((k,\kappa )+Q_{2d,1})}
\end{align*}
Then
\begin{alignat}{2}
\nm {V_{\phi _0}^*G_1}{M^{p_0,q_0}_{(\omega _0)}}
&\asymp
\nm {P_{\phi _0}G_1}{\sfW (\omega _0,\ell ^{p_0,q_0})}
\lesssim
\nm {G_1}{\sfW (\omega _0,\ell ^{p_0,q_0})}
\notag
\\[1ex]
&\le
\nm {G_2}{\sfW (\omega _0,\ell ^{p_0,q_0})}
\asymp
\nm b{\ell ^{p_0,q_0}_{(\omega _0)}},
\label{Eq:ModCoorbitNormRel}
\intertext{and}
\nm {f_j}{M^{p_j,q_j}_{(\omega _j)}} &\asymp \nm {a_j}{\ell ^{p_j,q_j}_{(\omega _j)}}
\label{Eq:ModCoorbitNormRel2}
\end{alignat}
in view of \cite[Proposition 3.4]{Toft13} (see also Theorem 3.3 in \cite{GaSa}) and
Proposition \ref{Prop:ProjMapModCont}.

\par

By \eqref{Eq:ConvDefFormula1} we have
\begin{multline}\label{Eq:G(x,lambda)Est}
G_2(x,\lambda )
\le
\sum _{k_1,\dots ,k_N\in \zz d}\left ( \prod _{j=1}^Na_j(k_j,\lambda )\right )
(\chi _{k_1+Q}*\cdots \chi _{k_N+Q})(x)
\\[1ex]
\le
\sum _{k_1,\dots ,k_N\in \zz d} \left ( \prod _{j=1}^Na_j(k_j,\lambda )\right )
\chi _{k_1+\cdots +k_N+Q_{d,N}}(x).
\end{multline}
We observe that
$$
\chi _{k_1+\cdots +k_N+Q_{d,N}}(x)=0
\quad \text{when}\quad
x\notin l+Q_d,\ (k_1,\dots ,k_N)\notin \Omega _l,
$$
where
$$
\Omega _l
=
\sets {(k_1,\dots ,k_N)\in \zz {Nd}}{l_j-N \le k_{1,j}+\cdots +k_{N,j}\le l_j+1},
$$
and
$$
k_n=(k_{n,1},\dots ,k_{n,d})\in \zz d
\quad \text{and}\quad
l=(l_1,\dots ,l_d)\in \zz d,\qquad n=1,\dots ,N.
$$
Hence, if $x=l$ in \eqref{Eq:G(x,lambda)Est}, we get
\begin{multline}\label{Eq:GExprEst}
b(l,\lambda )
\le
\sum _{(k_1,\dots ,k_N)\in \Omega _l}\left ( \prod _{j=1}^Na_j(k_j,\lambda )\right )
\\[1ex]
\le
\sum _{m\in I_{N+1}}(a_1(\cdo ,\lambda )*\cdots *a_N(\cdo ,\lambda ))(l-Ne_0+m),
\end{multline}
where $e_0=(1,\dots ,1)\in \zz d$ and $I_N = \{ 0,\dots ,N\} ^d$.
By multiplying with $\omega _0(l,\lambda )$, using \eqref{Eq:WeightModConv},
the fact that $I_N$ is a finite set and that $\omega _0$ is moderate,
we obtain
\begin{multline*}
b(l,\lambda )\omega _0(l,\lambda )
\le
\sum _{m\in I_{N+1}}(a_1(\cdo ,\lambda )*\cdots *a_N(\cdo ,\lambda ))
(l-Ne_0+m)\omega _0(l,\lambda )
\\[1ex]
\le
\sum _{m\in I_{N+1}}(a_1(\cdo ,\lambda )*\cdots *a_N(\cdo ,\lambda ))
(l-Ne_0+m)\omega _0(l-Ne_0+m,\lambda )
\\[1ex]
\lesssim
\sum _{m\in I_{N+1}}\big ( (a_1(\cdo ,\lambda )\omega _1(\cdo ,\lambda ))
*\cdots *(a_N(\cdo ,\lambda ) \omega _N(\cdo ,\lambda )) \big )
(l-Ne_0+m).
\end{multline*}
Hence \eqref{Eq:GExprEst} gives
\begin{multline}\tag*{(\ref{Eq:GExprEst})$'$}
b_{\omega _0}(l,\lambda )
\lesssim
\sum _{(k_1,\dots ,k_N)\in \Omega _l}
\left ( \prod _{j=1}^Na_{j,\omega _j}(k_j,\lambda )\right )
\\[1ex]
=
\sum _{m\in I_{N+1}} ( a_{1,\omega _1}(\cdo ,\lambda )
*\cdots *a_{N,\omega _N}(\cdo ,\lambda )) 
(l-Ne_0+m),
\end{multline}
where
$$
a_{j,\omega _j}(k,\kappa ) = a_j(k,\kappa)\omega _j(k,\kappa )
\quad \text{and}\quad
b_{\omega _0}(k,\kappa ) = b(k,\kappa)\omega _0(k,\kappa ).
$$

\par

If we apply the $\ell ^{p_0}$ quasi-norm on \eqref{Eq:GExprEst}
with respect to the $l$ variable, then Proposition
\ref{Prop:HolderYoungDiscrLebSpaces} (2) and the fact that
$I_{N+1}$ is a finite set give
\begin{multline*}
\nm {b_{\omega _0}(\cdo ,\lambda )}{\ell ^{p_0}}
\lesssim
\NM {\sum _{m\in I_{N+1}} ( a_{1,\omega _1}(\cdo ,\lambda )
*\cdots *a_{N,\omega _N}(\cdo ,\lambda )) 
(\cdo -Ne_0+m)}{\ell ^{p_0}}
\\[1ex]
\lesssim
\sum _{m\in I_{N+1}} \nm {( a_{1,\omega _1}(\cdo ,\lambda )
*\cdots *a_{N,\omega _N}(\cdo ,\lambda )) 
(\cdo -Ne_0+m)}{\ell ^{p_0}}
\\[1ex]
\asymp
\nm {a_{1,\omega _1}(\cdo ,\lambda )
*\cdots *a_{N,\omega _N}(\cdo ,\lambda )}{\ell ^{p_0}}
\\[1ex]
\le
\nm {a_{1,\omega _1}(\cdo ,\lambda )}{\ell ^{p_1}}
\cdots
\nm {a_{N,\omega _N}(\cdo ,\lambda )}{\ell ^{p_N}}.
\end{multline*}
By applying the $\ell ^{q_0}$ quasi-norm and using
Proposition
\ref{Prop:HolderYoungDiscrLebSpaces} (1)
we now get
$$
\nm {b_{\omega _0}}{\ell ^{p_0,q_0}}
\lesssim
\nm {a_{1,\omega _1}}{\ell ^{p_1,q_1}}\cdots \nm {a_{N,\omega _N}}{\ell ^{p_N,q_N}}.
$$
This is the same as
$$
\nm {G_2}{L^{p_0,q_0}_{(\omega _0)}}
\lesssim
\nm {F_1}{L^{p_1,q_1}_{(\omega _1)}}\cdots \nm {F_N}{L^{p_N,q_N}_{(\omega _N)}}.
$$
A combination of this estimate with \eqref{Eq:ModCoorbitNormRel} and
\eqref{Eq:ModCoorbitNormRel2} gives that $f_1*\cdots *f_N$ is well-defined
and that \eqref{Eq:ConvMod1} holds.

\par

Next we prove \eqref{Eq:ConvMod2}.
Let $r=\min (1,q_0)$. Then \eqref{Eq:GExprEst}$'$ gives
\begin{multline*}
b_{\omega _0}(l,\lambda )^r
\lesssim
\sum _{(k_1,\dots ,k_N)\in \Omega _l}
\left ( \prod _{j=1}^Na_{j,\omega _j}(k_j,\lambda )^r\right )
\\[1ex]
=
\sum _{m\in I_{N+1}} ( a_{1,\omega _1}(\cdo ,\lambda )^r
*\cdots *a_{N,\omega _N}(\cdo ,\lambda )^r) 
(l-Ne_0+m).
\end{multline*}
By applying the $\ell ^{q_0/r}$ norm with respect to the $\lambda$
variable and using Minkowski's and H{\"o}lder's
inequalities we obtain
\begin{multline*}
\nm {b_{\omega _0}(l,\cdo )}{\ell ^{q_0}}^r
=
\nm {b_{\omega _0}(l,\cdo )^r}{\ell ^{q_0/r}}
\lesssim
\sum _{(k_1,\dots ,k_N)\in \Omega _l}
\NM
{\prod _{j=1}^Na_{j,\omega _j}(k_j,\cdo )^r}{\ell ^{q_0/r}}
\\[1ex]
\le
\sum _{(k_1,\dots ,k_N)\in \Omega _l}
\left ( \prod _{j=1}^N c_{j,\omega _j}(k_j)^r \right )
=
\sum _{m\in I_{N+1}} ( c_{1,\omega _1}^r
*\cdots *c_{N,\omega _N}^r) 
(l-Ne_0+m),
\end{multline*}
where
$$
c_{j,\omega _j}(k) = \nm {a_{j,\omega _j}(k,\cdo )^r}{\ell ^{q_j/r}}^{1/r}
=
\nm {a_{j,\omega _j}(k,\cdo )}{\ell ^{q_j}}.
$$

\par

An application of the $\ell ^{p_0/r}$ quasi-norm on the last inequality and
using Proposition \ref{Prop:HolderYoungDiscrLebSpaces} (2)
now gives
\begin{multline*}
\nm {b_{\omega _0}}{\ell ^{q_0,p_0}_*}^r
\lesssim
\sum _{m\in I_{N+1}}  \nm {(c_{1,\omega _1}^r
*\cdots *c_{N,\omega _N}^r) (\cdo -Ne_0+m)}{\ell ^{p_0/r}}
\\[1ex]
\asymp
\nm {c_{1,\omega _1}^r
*\cdots *c_{N,\omega _N}^r}{\ell ^{p_0/r}}
\le
\nm {c_{1,\omega _1}^r}{\ell ^{p_1/r}}
\cdots 
\nm {c_{N,\omega _N}^r}{\ell ^{p_N/r}}
\\[1ex]
=
\big (\nm {c_{1,\omega _1}}{\ell ^{p_1}}
\cdots 
\nm {c_{N,\omega _N}}{\ell ^{p_N}} \big )^r,
\end{multline*}
which is the same as
$$
\nm {G_2}{L^{p_0,q_0}_{*,(\omega _0)}}
\lesssim
\nm {F_1}{L^{p_1,q_1}_{*,(\omega _1)}}
\cdots
\nm {F_N}{L^{p_N,q_N}_{*,(\omega _N)}},
$$
which in particular shows that $f_1*\cdots *f_N$ is well-defined.
Since
$$
\nm {f_1*\cdots *f_N}{W^{p_0,q_0}_{(\omega _0)}}\asymp
\nm G{L^{p_0,q_0}_{*,(\omega _0)}}
\quad \text{and}\quad
\nm {f_j}{W^{p_j,q_j}_{(\omega _j)}}
\asymp
\nm {F_j}{L^{p_j,q_j}_{*,(\omega _j)}},\ j=1,\dots ,N,
$$
we get \eqref{Eq:ConvMod2}.

\par

We need to prove
the associativity, symmetry and invariance
with respect to $\phi _0,\dots ,\phi _N$ in Definition
\ref{Def:MultConvSTFT}. We observe that if
$$
r_j=\max (p_j,1)
\quad \text{and}\quad
s_j=\frac {q_j}q,\ q=\min _{0\le j\le N}(q_j),
\quad j=1,\dots ,N,
$$
then $M^{p_j,q_j}_{(\omega _j)}(\rr d)\subseteq M^{r_j,s_j}_{(\omega _j)}(\rr d)$,
$j=1,\dots ,N$. By straight-forward computations it follows that if
\eqref{Eq:LebExpYoungHolder1} or \eqref{Eq:LebExpYoungHolder2} hold, then
\eqref{Eq:LebExpYoungHolder1} respectively \eqref{Eq:LebExpYoungHolder2} still
hold with $r_j$ and $s_j$ in place of $p_j$ and $q_j$, respectively, $j=1,\dots ,N$,
for some $r_0,s_0\in [1,\infty ]$.
This reduce ourself
to the case when $p_j,q_j\in [1,\infty ]$ for every $j=0,\dots ,N$, in which case all
modulation spaces are Banach spaces. We observe that
Lemmas 5.2--5.4 and their proofs in \cite{Toft3} still hold true
when $\omega _j$ are allowed to belong to the class $\mascP _E(\rr {2d})$, provided
the involved window functions $\chi _j$ belong to $\Sigma _1(\rr d)$, and all distributions
are allowed to belong to $\Sigma _1'$ instead of $\mascS '$. 
The the associativity, symmetric assertions and
invariant properties with respect to the choice of $\phi _0,\dots ,\phi _N$ in Definition
\ref{Def:MultConvSTFT} now follows from these modified Lemmas 5.2--5.4 in 
\cite{Toft3}
and their proofs. This gives the results.
\end{proof}

\par

\begin{rem}\label{Rem:ConvModUniqCase}
Suppose that $p_1$, $q_j$ and $\omega _j$ are the same as in
Theorems \ref{Thm:MultMod1}--\ref{Thm:ConvMod2}, and that
$p_j+q_j=\infty$ for at most one $j\in \{ 1,\dots ,N\}$. Then it follows
that extensions of the mappings $(f_1,\dots ,f_N)\mapsto f_1\cdots f_N$
and $(f_1,\dots ,f_N)\mapsto f_1*\cdots *f_N$ from 
$\Sigma _1(\rr d)\times \cdots \times \Sigma _1(\rr d)$ to
$\Sigma _1(\rr d)$ in Theorems \ref{Thm:MultMod1}--\ref{Thm:ConvMod2}
are \emph{unique}.

\par

In fact, by the proof of Theorem \ref{Thm:MultConvMod1} below,
we may assume that $p_j,q_j\ge 1$ for every $j$. If $p_j,q_j<\infty$ for every
$j\in \{1,\dots ,N\}$, then the uniquenesses
follow from \eqref{Eq:MultMod1}, \eqref{Eq:MultMod2}, \eqref{Eq:ConvMod1},
\eqref{Eq:ConvMod2} and the fact that $\Sigma _1(\rr d)$ is dense in each
$M^{p_j,q_j}_{(\omega _j)}(\rr d)$ and $W^{p_j,q_j}_{(\omega _j)}(\rr d)$
for $j\in \{ 1,\dots ,N\}$. For the general situation, the assertion follows from
the previous case and duality.
\end{rem}

\par

Evidently, Theorems \ref{Thm:MultMod1}--\ref{Thm:ConvMod2}
show that multiplications and convolutions on $\Sigma _1(\rr d)$
can be extended to involve suitable quasi-Banach modulation spaces.
Remark \ref{Rem:ConvModUniqCase} shows that in most situations,
these extensions from products on $\Sigma _1(\rr d)$ are unique.
For the multiplication and convolution mappings in
Theorems \ref{Thm:MultMod1} and \ref{Thm:ConvMod2} we can say more.

\par

\begin{thm}\label{Thm:MultConvMod1}
Let $\omega _j\in \mascP _E(\rr {2d})$
and $p_j, q_j\in (0,\infty ]$ and $j\in \{0,\dots ,N\}$. Then the following is true:
\begin{enumerate}
\item if \eqref{Eq:WeightModMult},
\eqref{Eq:LebExpHolderYoung1} and \eqref{Eq:QrFuncDef}
hold, then
$(f_1,\dots ,f_N)\mapsto f_1\cdots f_N$ 
from $\Sigma _1(\rr d)\times \cdots \times \Sigma _1(\rr d)$
to $\Sigma _1(\rr d)$ is uniquely extendable to
a continuous map from $M^{p_1,q_1}_{(\omega _1)}(\rr d)
\times \cdots \times M^{p_N,q_N}_{(\omega _N)}(\rr d)$
to $M^{p_0,q_0}_{(\omega _0)}(\rr d)$, and
\eqref{Eq:MultMod1} holds;

\vrum

\item if \eqref{Eq:WeightModConv}, \eqref{Eq:QrFuncDef} and
\eqref{Eq:LebExpYoungHolder1} hold, then
$(f_1,\dots ,f_N)\mapsto f_1*\cdots *f_N$ 
from $\Sigma _1(\rr d)\times \cdots \times \Sigma _1(\rr d)$
to $\Sigma _1(\rr d)$ is uniquely extendable to
a continuous map from $W^{p_1,q_1}_{(\omega _1)}(\rr d)
\times \cdots \times W^{p_N,q_N}_{(\omega _N)}(\rr d)$
to $W^{p_0,q_0}_{(\omega _0)}(\rr d)$, and
\eqref{Eq:ConvMod2} holds;
\end{enumerate}
\end{thm}

\par

The problems with uniqueness in Theorem \ref{Thm:MultConvMod1}
appear when one or more Lebesgue
exponents are equal to infinity, since $\Sigma _1(\rr d)$ fails to be dense
in corresponding modulation spaces. In these situations we shall use
\emph{narrow convergence}, introduced in \cite{Sjo},
and is a weaker form of convergence than the norm convergence.

\par

\begin{defn}\label{Def:NarrowConv}
Let $\omega \in \mascP _E(\rr {2d})$, $p,q\in [1,\infty ]$,
$f,f_j\in M^{p,q}_{(\omega )}(\rr d)$, $j\ge 1$ and let
$$
H_{f,\omega ,p}(\xi )\equiv \nm {V_\phi f(\cdo ,\xi )\omega (\cdo ,\xi )}{L^p(\rr d)}.
$$
Then $f_j$ is said to converge to $f$
\emph{narrowly} as $j\to \infty$, if the following conditions are fulfilled:
\begin{enumerate}
\item $f_j\to f$ in $\Sigma _1'(\rr d)$ as $j\to \infty$;

\vrum

\item $H_{f_j,\omega ,p}\to H_{f,\omega ,p}$ in $L^q(\rr d)$ as $j\to \infty$.
\end{enumerate}
\end{defn}

\par

The following result is a special case of Theorem 4.17 in \cite{Toft10}. The proof
is therefore omitted.

\par

\begin{prop}\label{Prop:NarrowConv}
Let $\omega \in \mascP _E(\rr {2d})$ and $p,q\in [1,\infty ]$ be
such that $q<\infty$. Then
$\Sigma _1(\rr d)$ is dense in $M^{p,q}_{(\omega )}(\rr d)$
with respect to the narrow convergence.
\end{prop}

\par

We also need the following generalization of Lebesgue's theorem,
which follows by a straight-forward application of Fatou's lemma.

\par

\begin{lemma}\label{Lemma:GenLeb}
Let $\mu$ be a positive measure on a measurable set $\Omega$,
$\{ f_j \}_{j=1}^\infty$ and $\{ g_j \}_{j=1}^\infty$ be sequences in
$L^1(d\mu )$ such that $f_j\to f$ a.{\,}e., $g_j\to g$ in $L^1(d\mu )$
as $j$ tends to infinity, and that $|f_j|\le g_j$ for every $j\in \mathbf N$.
Then $f_j\to f$ in in $L^1(d\mu )$ as $j$ tends to infinity.
\end{lemma}

\par

\begin{rem}\label{Rem:NarrowConv}
The narrow convergence is especially interesting when $p=\infty$. 
Let $\omega \in \mascP _E(\rr {2d})$, $q\in [1,\infty )$, $\phi \in \Sigma _1(\rr d)$
and $f\in M^{\infty ,q}_{(\omega )}(\rr d)$,
$f_j\in \Sigma _1(\rr d)$ converges to $f$ narrowly
as $j\to \infty$, and let $H_{f,\omega ,\infty}$ be the same as in
Definition \ref{Def:NarrowConv}. Then we may choose these $f_j$ such that
\begin{equation}\label{Eq:RemNarrowConv}
\begin{gathered}
\lim _{j\to \infty}V_\phi f_j(x,\xi )= V_\phi f(x,\xi ),
\quad
|V_\phi f(x,\xi )\omega (x,\xi )| \le H_{f,\omega ,\infty}(\xi )
\\[1ex]
\text{and}\quad
\lim _{j\to \infty}
\nm {H_{f_j,\omega ,\infty} - H_{f,\omega ,\infty}}{L^q}= 0
\end{gathered}
\end{equation}
(See \cite[Theorem 4.17]{Toft10} and its proof.)
It is then possible to apply
Lemma \ref{Lemma:GenLeb} in integral expressions containing
$V_\phi f_j(x,\xi )$ and $V_\phi f(x,\xi )$ and perform suitable
limit processes.
\end{rem}

\par

\begin{proof}[Proof of Theorem \ref{Thm:MultConvMod1}]
Since (2) is the Fourier transform of (1), it suffices to prove (1).

\par

The existence of the extension follows from Theorem \ref{Thm:MultMod1}.
Since $M^{p,q}_{(\omega )}(\rr d)$ increases with $p$ and $q$, we may
assume that equality is attained in \eqref{Eq:LebExpHolderYoung1}
and that $p_0=\cdots p_N=\infty$. By replacing $q_j$ with
$$
r_j=\max (1,q_j),
$$
it follows from \eqref{Eq:LebExpHolderYoung1} that
for $r_0=\frac {q_0}q\ge 1$ and some $r_0\ge 1$,
$$
\frac 1{r_0}\le \sum _{j=1}^N \frac 1{r_j} -N+1,
$$
and that
$$
M^{\infty ,q_j}_{(\omega _j)}(\rr d) \subseteq M^{\infty ,r_j}_{(\omega _j)}(\rr d).
$$
Suppose $g_1,g_2\in M^{\infty ,q_j}_{(\omega _j)}(\rr d)$ are such that
$g_1$ equals $g_2$ as elements in $M^{\infty ,r_j}_{(\omega _j)}(\rr d)$. Then
$g_1$ is also equal to $g_2$ as elements in $M^{\infty ,q_j}_{(\omega _j)}(\rr d)$.
Hence it suffices to prove the uniqueness of the product
$f_1\cdots f_N \in M^{\infty ,q_0}_{(\omega _0)}(\rr d)$ 
of $f_j\in M^{\infty ,q_j}_{(\omega _j)}(\rr d)$, $j=1,\dots ,N$,
when additionally $q_j\ge 1$, i.{\,}e.,
\begin{equation}\label{Eq:Youngq}
\frac 1{q_1}+\cdots +\frac 1{q_N} = N-1+\frac 1{q_0},\qquad q_0,\dots ,q_N\in [1,\infty ].
\end{equation}
In particular, all involved modulation spaces are Banach spaces.

\par

Let $j_0\in \{ 1,\dots ,N\}$ be chosen such that $q_j\le q_{j_0}$ for every
$j\in \{ 1,\dots ,N\}$. Then $j<\infty$ when $j\neq j_0$.

\par

The product $f_1\cdots f_N$ is uniquely defined and can be obtained through
\eqref{Eq:MultDefFormula1} for every $\fy \in \Sigma _1(\rr d)$ when
$f_j\in \Sigma _1(\rr d)$ and $f_{j_0}\in M^{\infty ,q_{j_0}}_{(\omega _{j_0})}(\rr d)$,
$j\neq j_0$. For general $f_j\in M^{\infty ,q_{j}}_{(\omega _{j})}(\rr d)$, choose
$f_{j,k}\in \Sigma _1(\rr d)$, $k=1,2,\dots $ such that $f_{j,k}$ converges to
$f_j$ narrowly as $k$ tends to infinity, and that \eqref{Eq:RemNarrowConv}
holds with $f_j$ and $f_{j,k}$ in place of $f$ and $f_j$, respectively. Then
it follows by replacing $f_j$ by $f_{j,k}$ when $j\neq j_0$ in \eqref{Eq:MultDefFormula1}
and applying Lemma \ref{Lemma:GenLeb} on the integral in \eqref{Eq:MultDefFormula1}
that
$$
\ell (\fy )\equiv \lim _{k\to \infty}(f_{1,k}\cdots f_{N,k},\fy )
$$
exists and defines an element in $f\in \Sigma _1'(\rr d)$. This shows that
the only possibility to define $f_1\cdots f_N$ in a continuous way is to put
$f_1\cdots f_N=f$, and the asserted uniqueness follows.
\end{proof}

\par

\section{Extensions and variations}\label{sec4}

\par

In this section we extend the results on step and Fourier step multipliers
to certain so-called curve step and Fourier curve step multipliers.
That is a generalized form of of step and Fourier step multipliers,
where the constants $a_0(j)$ in the definition of $M_{b,a_0}$ and
$M_{\mascF \! ,b,a_0}$ are replaced by
certain non-constant functions or even distributions.
In the end we are able to generalize Theorems \ref{Thm:Mainthm1}
and \ref{Thm:Mainthm2} to such multipliers. These achievements
are based on H{\"o}lder-Young relations
for multiplications and convolutions in Section \ref{sec3}. In the case
of trivial weights and all modulation spaces are Banach spaces,
our results are similar to \cite[Theorem 6]{BenGraGroOko} and
\cite[Proposition 4.12]{RSTT}.

\par

The multipliers and Fourier multipliers which we consider are given in
the following.

\par

\begin{defn}\label{Def:SlopeStepMult}
Let  $b\in \rr d_+$ be fixed, $\Lambda _b$ and $Q_b$ be given by
\eqref{Eq:LatticebDef} and \eqref{Eq:QubebDef},
and let
\begin{equation}\label{Eq:DefFuncSeq}
a_0 \equiv \{ a_0(j,\cdo )\} _{j\in \Lambda _b}\subseteq C^\infty (\rr d)
\end{equation}
be such that
$$
\Big (\sum _{j\in \Lambda _b}a_0(j,\cdo )\chi _{j+Q_b}\Big )
\in \Sigma _1'(\rr d).
$$
Then the multiplier
\begin{equation}\tag*{(\ref{Eq:DefStepMult})$'$}
M_{b,a_0}\, :\, f\mapsto \sum _{j\in \Lambda _b}a_0(j,\cdo )\chi _{j+Q_b}f,
\end{equation}
from $C^\infty (\rr d)$ to $L^\infty _{loc}(\rr d)$ is called
\emph{slope step multiplier} with respect to $b$ and $a_0$.
The Fourier multiplier
\begin{equation}\tag*{(\ref{Eq:DefStepFourMult})$'$}
M_{\mascF \! ,b,a_0} \equiv \mascF ^{-1}\circ M_{b,a_0} \circ \mascF ,
\end{equation}
is called \emph{slope step Fourier multiplier} with respect to $b$ and $a_0$.
\end{defn}

\par

First we perform some studies of
\begin{equation}\label{Eq:DefFuncFromSeq}
T_\psi a_0 \equiv \sum _{j\in \Lambda _b}a_0(j,\cdo )\psi (\cdo -j),
\end{equation}
where $\psi \in \mascS (\rr d)$ is suitable. The conditions on the sequence
\eqref{Eq:DefFuncSeq} that we have in mind are that for fixed
$\omega _0\in \mascP _E(\rr d)$ and $p\in (0,\infty ]$, the functions
\begin{equation}\label{Eq:ConsFuncSeq1}
\fkb _{a_0,\alpha} (x)
\equiv
\sup _{\beta \le \alpha}\left (\sup _{j\in \Lambda _b}
\left |(\partial _x^\beta a_0)(j,x)\right | \right )
\end{equation}
should belong to $L^p(\rr d)$ for every $\alpha \in \nn d$, or that
for some or for every $h>0$, the function
\begin{equation}\label{Eq:ConsFuncSeq2}
\fkb _{a_0,h} (x)
\equiv
\sup _{\alpha \in \nn d}\left (\sup _{j\in \Lambda _b}
\left |\frac {(\partial _x^\alpha a_0)(j,x)}{h^{|\alpha |}\alpha !^\sigma}
\right |\right )
\end{equation}
should belong to $L^p(\rr d)$.

\par

\begin{prop}\label{Prop:FuncFromSeq}
Let $b\in \rr d_+$ be fixed, $\Lambda _b$ be given by
\eqref{Eq:LatticebDef}, $s,h_0,\sigma >0$,
\eqref{Eq:DefFuncSeq}
be a sequence of functions on $C^\infty (\rr d)$, $\psi \in \mascS (\rr d)$,
and let $T_\psi a_0$, $\fkb _{a_0,\alpha}$ and $\fkb _{a_0,h}$
be given by \eqref{Eq:DefFuncFromSeq}--\eqref{Eq:ConsFuncSeq2}
when $\alpha \in \nn d$ and $h>0$. Then the following is true:
\begin{enumerate}
\item if $\fkb _{a_0,\alpha} \in L^\infty _{loc}(\rr d)$ for $\alpha =(0,\dots ,0)\in \nn d$,
then the series in \eqref{Eq:DefFuncFromSeq} is locally uniformly convergent
and defines an element in $C(\rr d)$;

\vrum

\item if $\fkb _{a_0,\alpha} \in L^\infty _{loc}(\rr d)$ for every $\alpha \in \nn d$, then
$T_\psi a_0\in C^\infty (\rr d)$ and
$$
|(\partial ^\alpha T_\psi a_0)(x)| \lesssim \fkb _{a_0,\alpha} (x),\qquad x\in \rr d,
$$
for every $\alpha \in \nn d$;

\vrum

\item if in addition $\psi \in \maclS _s^\sigma (\rr d)$ and
$\fkb _{a_0,h_0}\in L^\infty _{loc}(\rr d)$,
then
$$
|(\partial ^\alpha T_\psi a_0)(x)| \lesssim h^\alpha \alpha !^\sigma \fkb _{a_0,h_0}(x),
\qquad x\in \rr d,
$$
for some $h>0$;

\vrum

\item  if in addition $\psi \in \Sigma _s^\sigma (\rr d)$, $c>1$ and
$\fkb _{a_0,h_0}\in L^\infty _{loc}(\rr d)$,
then
$$
|(\partial ^\alpha T_\psi a_0)(x)|
\lesssim (ch_0)^\alpha \alpha !^\sigma \fkb _{a_0,h_0}(x), \qquad x\in \rr d.
$$
\end{enumerate}
\end{prop}

\par

\begin{proof}
We only prove (1) and (4). The other assertions follow by similar arguments and
are left for the reader.

\par

Let $\Lambda =\Lambda _b$, $\alpha =(0,\dots ,0)\in \nn d$ and suppose that
$\fkb _{a_0,\alpha} \in L^\infty _{loc}(\rr d)$. We have
$$
\sum _{j\in \Lambda}|a_0(j,x)\psi (x-j)|
\le
\fkb _{a_0,\alpha} (x) \sum _{j\in \Lambda}|\psi (x-j)|
\asymp
\fkb _{a_0,\alpha} (x),
$$
which shows that \eqref{Eq:DefFuncFromSeq} is locally uniformly convergent.
Since $a_0(j,\cdo )$ and $\psi (\cdo -j)$ are continuous functions, it follows that
$T_\psi a_0$ in \eqref{Eq:DefFuncFromSeq} is continuous.

\par

Next suppose additionally that $\psi \in \Sigma _s^\sigma (\rr d)$ and
consider $f=T_\psi a_0$. For every
$\alpha \in \nn d$, $\ep >0$ and $r>0$, we have
\begin{multline*}
|(\partial ^\alpha f)(x)|
\le
\sum _{j\in \Lambda}\sum _{\gamma \le \alpha}{{\alpha}\choose {\gamma}}
|\partial ^{\alpha -\gamma}a_0(j,x)||\partial ^\gamma \psi (x-j)|
\\[1ex]
\lesssim
\fkb _{a_0,h_0}(x)
\sum _{j\in \Lambda}\sum _{\gamma \le \alpha}{{\alpha}\choose {\gamma}}
h_0^{|\alpha -\gamma|}(\alpha -\gamma)!^\sigma \ep ^{|\gamma|}\gamma !^\sigma
e^{-r|x-j|^{\frac 1\sigma}}
\\[1ex]
\le
(h_0+\ep )^{|\alpha |}\alpha !^\sigma \fkb _{a_0,h_0}(x)
\sum _{j\in \Lambda}e^{-r|x-j|^{\frac 1\sigma}}
\asymp
(h_0+\ep )^{|\alpha |}\alpha !^\sigma \fkb _{a_0,h_0}(x),
\end{multline*}
and the result follows.
\end{proof}

\par

In the next result we show that if $\fkb _{a_0,\alpha}$ or $\fkb _{a_0,h}$ in the
previous proposition belong to $\sfW ^1(\omega _0,\ell ^{p})$,
then for $T_\psi a_0$ in \eqref{Eq:DefFuncFromSeq} we have
\begin{equation}\label{Eq:FunctionInModspace}
T_\psi a_0\in M^{p,q}_{(\omega _r)}(\rr d)\bigcap W^{p,q}_{(\omega _r)}(\rr d),
\end{equation}
and
\begin{align}
\nm {T_\psi a_0}{M^{p,q}_{(\omega _r)}}
+
\nm {T_\psi a_0}{W^{p,q}_{(\omega _r)}}
&\lesssim
\max _{|\alpha |\le N} \nm {\fkb _{a_0,\alpha}}{\sfW ^1(\omega _0,\ell ^{p})}.
\label{Eq:SeqWienerEst1}
\intertext{or}
\nm {T_\psi a_0}{M^{p,q}_{(\omega _r)}}
+
\nm {T_\psi a_0}{W^{p,q}_{(\omega _r)}}
&\lesssim
\nm {\fkb _{a_0,h}}{\sfW ^1(\omega _0,\ell ^{p})}
\label{Eq:SeqWienerEst2}
\end{align}
See also Remark \ref{Rem:WienerSpaces}
for notations.

\par

\begin{prop}\label{Prop:HomeSymbSlopeMult}
Let $\omega _0\in \mascP _E(\rr d)$, $\Lambda \subseteq \rr d$,
$s$, $h_0$, $\sigma >0$, $T_\psi a_0$, $\fkb _{a_0,\alpha}$,
$\fkb _{a_0,h}$ and $\psi$ be the same as in Proposition \ref{Prop:FuncFromSeq},
and let $p,q\in (0,\infty ]$.
Then the following is true:
\begin{enumerate}
\item if in addition $\omega _0\in \mascP (\rr d)$,
$\fkb _{a_0,\alpha} \in \sfW ^1(\omega _0,\ell ^{p})$
for every $\alpha \in \nn d$,
and $\vartheta _r(x,\xi )= \omega _0(x)\eabs \xi ^r$ when $r\ge 0$, then
\eqref{Eq:FunctionInModspace}
and \eqref{Eq:SeqWienerEst1} hold;

\vrum

\item if in addition $\omega _0\in \mascP _{E,s}^0(\rr d)$,
$\psi \in \maclS _s^\sigma (\rr d)$,
$\fkb _{a_0,h} \in \sfW ^1(\omega _0,\ell ^{p})$
for some $h>0$,
and $\vartheta _r(x,\xi )= \omega _0(x) e^{r|\xi |^{\frac 1\sigma}}$
when $r\ge 0$, then \eqref{Eq:FunctionInModspace} and
\eqref{Eq:SeqWienerEst2} hold for some $r>0$;

\vrum

\item if in addition $\omega _0\in \mascP _{E,s}(\rr d)$,
$\psi \in \Sigma _s^\sigma (\rr d)$,
$\fkb _{a_0,h} \in \sfW ^1(\omega _0,\ell ^{p})$
for every $h>0$,
and $\vartheta _r(x,\xi )= \omega _0(x) e^{r|\xi |^{\frac 1\sigma}}$
when $r\ge 0$, then
\eqref{Eq:FunctionInModspace} and \eqref{Eq:SeqWienerEst2}
hold for every $r>0$.
\end{enumerate}
\end{prop}

\par

\begin{proof}
We only prove (2). The other assertions follow by similar arguments and
are left for the reader.

\par

Let $f=T_\psi a_0$, $\psi _j=\psi (\cdo -j)$ and
$\widetilde \phi (x)= \overline{\phi (-x)}$. Since
$$
(V_\phi f)(x,\xi ) = e^{-i\scal x\xi} (V_{\widehat \phi}\widehat f)(\xi ,-x),
$$
we get
\begin{multline*}
(V_\phi f)(x,\xi )
=
e^{-i\scal x\xi}
\sum _{j\in \Lambda} (V_{\widehat \phi}\mascF (\psi _ja_0(j,\cdo )))(\xi ,-x)
\\[1ex]
=
e^{-i\scal x\xi}
\sum _{j\in \Lambda} \mascF ^{-1}
(\mascF (\psi _ja_0(j,\cdo ))\cdot
\overline {\widehat \phi (\cdo -\xi )})(x)
\\[1ex]
=
(2\pi )^{-\frac d2}e^{-i\scal x\xi}
\sum _{j\in \Lambda} 
\big ( (\psi _ja_0(j,\cdo ))*
(\widetilde \phi \cdot e^{i\scal \cdo \xi}) \big )(x).
\end{multline*}
Hence, Leibnitz rule, integrations by parts and Proposition \ref{Prop:GSexpcond}
give
\begin{multline*}
|\xi ^\alpha (V_\phi f)(x,\xi )|
\lesssim
\sum _{j\in \Lambda} 
\left | \big ( (\psi _ja_0(j,\cdo ))*
(\widetilde \phi \cdot (D_x ^\alpha e^{i\scal \cdo \xi})) \big )(x) \right |
\\[1ex]
\le
\sum
\frac {\alpha !}{\gamma _1!\gamma _2!\gamma _3!}
\big ( |(\partial ^{\gamma _1}\psi _j)(\partial ^{\gamma _2}a_0(j,\cdo ))|*
|\partial ^{\gamma _3}\widetilde \phi | \big )(x)
\\[1ex]
\lesssim
\sum
3^{|\alpha |}h^{|\gamma _1+\gamma _2+\gamma _3|}
(\gamma _1!\gamma _2!\gamma _3!)^\sigma
\big ( (e^{-r|\cdo -j|^{\frac 1s}} \fkb _{a_0,h})*e^{-r|\cdo |^{\frac 1s}}
\big )(x)
\\[1ex]
\le
(9h)^{|\alpha |}\alpha !^\sigma \sum _{j\in \Lambda}
\big ( (e^{-r|\cdo -j|^{\frac 1s}} \fkb _{a_0,h})*e^{-r|\cdo |^{\frac 1s}}
\big )(x)
\\[1ex]
\asymp
(9h)^{|\alpha |}\alpha !^\sigma
\big ( \fkb _{a_0,h}*e^{-r|\cdo |^{\frac 1s}} \big )(x).
\end{multline*}
Here the second and third sums are taken with respect to
all $j\in \Lambda$ and all $\gamma _1,\gamma _2,\gamma _3\in \nn d$
such that $\gamma _1+\gamma _2+\gamma _3=\alpha$.

\par

This implies that for some constant $C$ which is independent of $h$, $r$
and $\alpha$ we have
$$
\left (
\frac {|\xi |^{\frac 1\sigma}}{(Ch)^{\frac 1\sigma}}
\right )^k |V_\phi f(x,\xi )|^{\frac 1\sigma}
\lesssim
2^{-k} \left (\big ( \fkb _{a_0,h}*e^{-r|\cdo |^{\frac 1s}} \big )(x)\right )^{\frac 1\sigma},
$$
and by taking the sum over all $k\ge 0$ we land on
\begin{equation*}
|V_\phi f(x,\xi )e^{r_h|\xi |^{\frac 1\sigma}}|
\lesssim
\big ( \fkb _{a_0,h}*e^{-r|\cdo |^{\frac 1s}} \big )(x),
\qquad
r_h = \frac \sigma{(Ch)^{\frac 1\sigma}}.
\end{equation*}
By multiplying with $\omega _0$ and using that
$\omega _0(x+y)\lesssim \omega _0(x)e^{r|y|^{\frac 1s}}$ for every $r>0$,
we obtain 
\begin{equation}\label{Eq:EstExpSTFT}
|V_\phi f(x,\xi )\vartheta _{r_h}(x,\xi )|
\lesssim
\big ( (\fkb _{a_0,h}\omega _0)*e^{-r_0|\cdo |^{\frac 1s}} \big )(x),
\qquad
r_h = \frac \sigma{(Ch)^{\frac 1\sigma}},
\end{equation}
for some $r_0>0$. 

\par

By applying \cite[Proposition 2.5]{Toft13} on the last inequality we obtain
$$
\nm {V_\phi f}{\sfW (\vartheta _{r_h},\ell ^{p,\infty })}
+
\nm {V_\phi f}{\sfW (\vartheta _{r_h},\ell ^{p,\infty }_*)}
\lesssim
\nm {e^{-r|\cdo |^{\frac 1s}}}{\sfW (1,\ell ^{\min (1,p)})}
\nm {\fkb _{a_0,h}}{\sfW (\vartheta _{r_h},\ell ^{p})}.
$$
The result now follows for general $q\in (0,\infty ]$ from the relations
$$
\nm {V_\phi f}{\sfW (\vartheta _{r_h},\ell ^{p,\infty })}
\asymp \nm f{M^{p,\infty}},
\quad
\text{and}\quad
M^{p,\infty}_{(\vartheta _{2r})}(\rr d) \hookrightarrow
M^{p,q}_{(\vartheta _{r})}(\rr d) \hookrightarrow 
M^{p,\infty}_{(\vartheta _{r})}(\rr d),
$$
and similarly with $W^{p,q}_{(\omega )}$ and $\ell ^{p,q}_*$ spaces
in place of $M^{p,q}_{(\omega )}$ and $\ell ^{p,q}$ spaces.
%
%
%
%
\end{proof}

\par

We have now the following extension of Theorem \ref{Thm:Mainthm1}. Here
involved Lebesgue exponents and weight functions should fullfil
\begin{equation}\label{Eq:MainThmCondLebExpAgain}
\frac 1{p_2}-\frac 1{p_1}\le \frac 1p,
\quad
\frac 1{q_1}-\frac 1{q_2} \ge \max \left ( \frac 1{p_2}-1,0 \right ), 
\end{equation}
and
\begin{equation}\label{Eq:WeightCondUsual}
\omega _{0,2}(x)\lesssim \omega _{0,1}(x)\omega _0(x).
\end{equation}

\par

\begin{thm}\label{Thm:Mainthm1Ext}
Let $p,p_1,p_2\in (0,\infty ]$, $q\in (1,\infty )$, $q_1,q_2\in (\min (1,p_2),\infty )$
be such that \eqref{Eq:MainThmCondLebExpAgain} holds,
$b>0$, $\omega _0,\omega _{0,j}$ be weights on $\rr d$ such that
\eqref{Eq:WeightCondUsual} holds true,
$\omega (x,\xi )=\omega _0(x)$ and $\omega _j(x,\xi )=\omega _{0,j}(x)$,
$j=1,2$, $x,\xi \in \rr d$. Let $a_0$ in \eqref{Eq:DefFuncSeq}
be such that $\Lambda =\Lambda _b$ and $a_0(j,\cdo )\in C^\infty (\rr d)$ for every
$j\in \Lambda _b$, and let $\fkb _{a_0,\alpha}$ and $\fkb _{a_0,h}$ be given by
\eqref{Eq:ConsFuncSeq1} and \eqref{Eq:ConsFuncSeq2}. Also
suppose that one of the following conditions hold true:
\begin{itemize}
\item[{\rm{(i)}}] $\fkb _{a_0,\alpha} \in \sfW ^1(\omega _0,\ell ^{p})$ for every $\alpha
\in \nn d$, and $\omega _0,\omega _{0,j}\in \mascP (\rr d)$, $j=1,2$;

\vrum

\item[{\rm{(ii)}}] $\fkb _{a_0,h} \in \sfW ^1(\omega _0,\ell ^{p})$ for some $h>0$,
and $\omega _0,\omega _{0,j}\in \mascP _{E,s}^0(\rr d)$, $j=1,2$;

\vrum

\item[{\rm{(iii)}}] $\fkb _{a_0,h} \in \sfW ^1(\omega _0,\ell ^{p})$ for every $h>0$,
and $\omega _0,\omega _{0,j}\in \mascP _{E,s}(\rr d)$, $j=1,2$.
\end{itemize}
Then the following is true:
\begin{enumerate}
\item $M_{b,a_0}$ is continuous from $W^{p_1,q}_{(\omega _1)}(\rr d)$ to
$W^{p_2,q}_{(\omega _2)}(\rr d)$;

\vrum

\item  $M_{b,a_0}$ is continuous from $M^{p_1,q_1}_{(\omega _1)}(\rr d)$
to $M^{p_2,q_2}_{(\omega _2)}(\rr d)$.
\end{enumerate}
\end{thm}

\par

\begin{proof}
We only prove the result when (iii) holds. The other cases follow by
similar arguments and is left for the reader.

\par

Let $\psi \in \Sigma _s^\sigma (\rr d)$ be such that $\psi =1$ on $Q_b$ and
supported in a neighbourhood of $Q_b$, and let $\Lambda _1,\dots ,\Lambda _N$
be sublattices of $\Lambda =\Lambda _b$ such that
$$
\bigcup _{j=1}^N\Lambda _j =\Lambda
\quad \text{and}\quad
\supp \psi (\cdo -{k_1})\bigcap \supp \psi (\cdo -{k_2}) = \emptyset
, k_1,k_2\in \Lambda _j,\ k_1\neq k_2,
$$
for every $j=1,\dots ,N$. Then
$$
M_{b,a_0} = \sum _{j=1}^N S_j,
$$
where $S_j=S_{2,j}\circ S_{1,j}$, with $S_{1,j}$ and $S_{2,j}$
being the multiplication operators with the functions
$$
\fy _{1,j} \equiv \sum _{k\in \Lambda _j}a_0(k,\cdo )\psi (\cdo -k)
\quad \text{and}\quad
\fy _{2,j} \equiv \sum _{k\in \Lambda _j}\chi _{Q_b}(\cdo -k),
$$
respectively. The result follows if we prove the asserted continuity
properties for $S_j$ in place of $M_{b,a_0}$ 

\par

By Proposition \ref{Prop:HomeSymbSlopeMult} it follows that
$\fy _{1,j} \in M^{p,q}_{(\vartheta _r)}(\rr d)\cap W^{p,q}_{(\vartheta _r)}(\rr d)$
for every $q\in (0,1]$ and $r>0$.
Hence, if we choose $q$ small enough, Theorems \ref{Thm:MultMod1} and
\ref{Thm:MultMod2} show that $S_{1,j}$ is continuous
from $W^{p_1,q}_{(\omega _1)}(\rr d)$ to $W^{p_2,q}_{(\omega _2)}(\rr d)$,
and from
$M^{p_1,q_1}_{(\omega _1)}(\rr d)$ to $M^{p_2,q_1}_{(\omega _2)}(\rr d)$.
In view of Theorem \ref{Thm:Mainthm1} one has that $S_{2,j}$ is continuous
on $W^{p_2,q}_{(\omega _2)}(\rr d)$, and from
$M^{p_2,q_1}_{(\omega _2)}(\rr d)$ to $M^{p_2,q_2}_{(\omega _2)}(\rr d)$,
for every $j$. By combining these mapping properties it follows that
$S_j$ is continuous from $W^{p_1,q}_{(\omega _1)}(\rr d)$ to
$W^{p_2,q}_{(\omega _2)}(\rr d)$, and from
$M^{p_1,q_1}_{(\omega _1)}(\rr d)$ to $M^{p_2,q_2}_{(\omega _2)}(\rr d)$
for every $j$, and the result follows.
\end{proof}

\par

By Fourier transforming the latter result we obtain the following
extension of Theorem \ref{Thm:Mainthm2}. The details are left for
the reader. Here
\begin{equation}\label{Eq:MainThmCondLebExpAgain2}
\frac 1{q_2}-\frac 1{q_1}\le \frac 1q
\quad \text{and} \quad
\frac 1{p_1}-\frac 1{p_2} \ge \max \left ( \frac 1{q_2}-1,0 \right ).
\end{equation}

\par

\begin{thm}\label{Thm:Mainthm2Ext}
Let $q,q_1,q_2\in (0,\infty ]$, $p\in (1,\infty )$, $p_1,p_2\in (\min (1,q_2),\infty )$
be such that \eqref{Eq:MainThmCondLebExpAgain2} holds,
$b>0$, $\omega _0,\omega _{0,j}$ be weights on $\rr d$ such that
\eqref{Eq:WeightCondUsual} holds true,
$\omega (x,\xi )=\omega _0(\xi )$ and $\omega _j(x,\xi )=\omega _{0,j}(\xi )$,
$j=1,2$, $x,\xi \in \rr d$. Let $f_0$ in \eqref{Eq:DefFuncSeq}
be such that $\Lambda =\Lambda _b$ and $a_0(j,\cdo )\in C^\infty (\rr d)$ for every
$j\in \Lambda _b$, and let $\fkb _{a_0,\alpha}$ and $\fkb _{a_0,h}$ be given by
\eqref{Eq:ConsFuncSeq1} and \eqref{Eq:ConsFuncSeq2}. Also
suppose that one of the following conditions hold true:
\begin{itemize}
\item[{\rm{(i)}}] $\fkb _{a_0,\alpha} \in \sfW ^1(\omega _0,\ell ^{p})$ for every $\alpha
\in \nn d$, and $\omega _0,\omega _{0,j}\in \mascP (\rr d)$, $j=1,2$;

\vrum

\item[{\rm{(ii)}}] $\fkb _{a_0,h} \in \sfW ^1(\omega _0,\ell ^{p})$ for some $h>0$,
and $\omega _0,\omega _{0,j}\in \mascP _{E,s}^0(\rr d)$, $j=1,2$;

\vrum

\item[{\rm{(iii)}}] $\fkb _{a_0,h} \in \sfW ^1(\omega _0,\ell ^{p})$ for every $h>0$,
and $\omega _0,\omega _{0,j}\in \mascP _{E,s}(\rr d)$, $j=1,2$.
\end{itemize}
Then the following is true:
\begin{enumerate}
\item $M_{\mascF,b,a_0}$ is continuous from $M^{p,q_1}_{(\omega _1)}(\rr d)$ to
$M^{p,q_2}_{(\omega _2)}(\rr d)$;

\vrum

\item  $M_{\mascF ,b,a_0}$ is continuous from $W^{p_1,q_1}_{(\omega _1)}(\rr d)$
to $W^{p_2,q_2}_{(\omega _2)}(\rr d)$.
\end{enumerate}
\end{thm}

\par

We observe that Theorems \ref{Thm:Mainthm1Ext} and \ref{Thm:Mainthm2Ext}
include the following extensions of Theorems \ref{Thm:Mainthm1} and
\ref{Thm:Mainthm2}.

\par

\begin{cor}\label{Cor:Mainthm1}
Let $p,p_1,p_2\in (0,\infty ]$, $q\in (1,\infty )$, $q_1,q_2\in (\min (1,p),\infty )$
be such that \eqref{Eq:MainThmCondLebExpAgain} hold,
$b>0$, $\omega _0,\omega _{0,j}\in \mascP _{E}(\rr {d})$
be such that \eqref{Eq:WeightCondUsual},
$\omega _j(x,\xi )=\omega _{0,j}(x)$, $j=1,2$, $x,\xi \in \rr d$,
$x,\xi \in \rr d$, and let $a_0\in \ell ^p_{(\omega _0)} (\Lambda _b)$.
Then the following is true:
\begin{enumerate}
\item $M_{b,a_0}$ is continuous from $W^{p_1,q}_{(\omega _1)}(\rr d)$ to
$W^{p_2,q}_{(\omega _2)}(\rr d)$;

\vrum

\item  $M_{b,a_0}$ is continuous from $M^{p_1,q_1}_{(\omega _1)}(\rr d)$
to $M^{p_2,q_2}_{(\omega _2)}(\rr d)$.
\end{enumerate}
\end{cor}

\par
%
%
%

\begin{cor}\label{Cor:Mainthm2}
Let $p\in (1,\infty )$, $p_1,p_2\in (\min (1,q),\infty )$, $q,q_1,q_2\in (0,\infty ]$
be such that \eqref{Eq:MainThmCondLebExpAgain2} hold,
$b>0$, $\omega _0,\omega _{0,j}\in \mascP _{E}(\rr {d})$
be such that \eqref{Eq:WeightCondUsual} holds,
$\omega _j(x,\xi )=\omega _{0,j}(\xi )$, $j=1,2$, $x,\xi \in \rr d$,
$x,\xi \in \rr d$, and let $a_0\in \ell ^q_{(\omega _0)} (\Lambda _b)$.
Then the following is true:
\begin{enumerate}
\item $M_{\mascF \! ,b,a_0}$ is continuous from $M^{p,q_1}_{(\omega _1)}(\rr d)$ to
$M^{p,q_2}_{(\omega _2)}(\rr d)$;

\vrum

\item  $M_{\mascF \! ,b,a_0}$ is continuous from $W^{p_1,q_1}_{(\omega _1)}(\rr d)$
to $W^{p_2,q_2}_{(\omega _2)}(\rr d)$.
\end{enumerate}
\end{cor}

\par

\begin{proof}[Proof of Corollaries \ref{Cor:Mainthm1} and \ref{Cor:Mainthm2}]
Let $\psi \in \Sigma _1^\sigma (\rr d)$ be compactly supported and
chosen such that $\psi =1$ on $Q_b$. Then the results follow by letting
$a_0(j,\cdo )= a_0(j)\psi (\cdo -j)$ in Theorems
\ref{Thm:Mainthm1Ext} and \ref{Thm:Mainthm2Ext}. The details are left
for the reader.
\end{proof}

\par

\appendix

\par

\section{}\label{AppA}

\par

In this appendix we give a proof of Proposition
\ref{Prop:HolderYoungDiscrLebSpaces} (2).

\par

\begin{proof}[Proof of Proposition \ref{Prop:HolderYoungDiscrLebSpaces} (2)]
Since $\ell ^p_{(\omega )}(\Lambda )$ increases with $p$ and
$\nm \cdo{\ell ^p_{(\omega )}}$ decreases with $p$, we may assume that
equality is attained in \eqref{Eq:LebExpYoung}.

\par

By \eqref{Eq:WeightYoung} we get
$$
|(a_1*\cdots *a_N)(k)\cdot \omega _0(k)|
\le
(|a_1\cdot \omega _1|*\cdots *|a_N\cdot \omega _N|)(k),
$$
provided the left-hand side makes sense. A combination of this inequality
with the fact that the map $a_j\mapsto a_j\cdot \omega _j$ is an isometric
bijection from $\ell ^{p_j}_{(\omega _j)}(\Lambda)$ to
$\ell ^{p_j}(\Lambda)$, reduces ourselves to the case when $\omega _j=1$
for every $j$.

\par

Let $I_N=\{ 1,\dots , N\}$ and $j_0\in I_N$ be chosen such that
$p_{j_0}\ge p_j$ for every $j\in I_N$. Then
\begin{equation}\label{Eq:YoungComp1}
\frac 1{p_0} = \frac 1{p_{j_0}} +\sum _{j\neq j_0} \left (
\frac 1{p_j}-\max \left ( 1,\frac 1{p_j}\right )
\right )
\le
\frac 1{p_{j_0}},
\end{equation}
giving that $p_j\le p_0$ for every $j\in I_N$.

\par

In particular, if $p_1,\dots ,p_N\ge 1$, then \eqref{Eq:LebExpYoung}
shows that $p_0\ge 1$, and the assertion agrees with the
usual Young's inequality. By \eqref{Eq:YoungComp1} it also follows
that if $p_{j_0}=\infty$, then $p_0=\infty$ and $p_j\le 1$ when
$j\in I_N \setminus \{ j_0 \}$, since otherwise the equality in
\eqref{Eq:YoungComp1} may not hold for some $p_0$ in $(0,\infty ]$.
In this case we have that if $a_j\in \ell ^{p_j}(\Lambda )$, $j\in I_N$, then
the facts $\ell ^{p_j}\subseteq \ell ^1$ and
$\nm \cdo {\ell ^1}\le \nm \cdo {\ell ^{p_j}}$ for $j\neq j_0$ give
\begin{multline*}
a_1*\cdots *a_N \in \ell ^{p_1}(\Lambda )*\cdots *\ell ^{p_{j_0}}(\Lambda )*
\cdots *\ell ^{p_N}(\Lambda )
\\[1ex]
\subseteq
\ell ^1(\Lambda )*\cdots *\ell ^\infty (\Lambda )* \cdots *\ell ^1(\Lambda )
= \ell ^\infty (\Lambda )
\end{multline*}
and
$$
\nm {a_1*\cdots *a_N}{\ell ^\infty} \le \nm {a_1}{\ell ^1}\cdots
\nm {a_{j_0}}{\ell ^\infty} \cdots \nm {a_N}{\ell ^1}
\le
\nm {a_1}{\ell ^{p_1}}\cdots
\nm {a_{j_0}}{\ell ^\infty} \cdots \nm {a_N}{\ell ^{p_N}},
$$
and the result follows in this case as well.

\par

It remains to prove the result when $p_j<1$ for at least one $j\in I_N$ and that
$p_{j_0}<\infty$, giving that $\ell _0$ is dense in $\ell ^{p_j}$ for every $j\in I_N$.
Hence the result follows if we prove \eqref{Eq:YoungEst} when
$a_j\in \ell _0$.

\par

Suppose $N=2$. Then $I=\{  j_0,j_1\}$ for some $j_1\in \{ 1,2 \}$ with $p_{j_1}<1$
and $j_1\le j_0$. Then
$$
\frac 1{p_0} = \frac 1{p_{j_0}} + \frac 1{p_{j_1}} - \frac 1{p_{j_1}} = \frac 1{p_{j_0}},
$$
i.{\,}e. $p_0=p_{j_0}$. This gives
$$
\nm {a_{j_0}*a_{j_1}}{\ell ^{p_0}} = \nm {a_{j_0}*a_{j_1}}{\ell ^{p_{j_0}}}
\le \nm {a_{j_0}}{\ell ^{p_{j_0}}}\nm {a_{j_1}}{\ell ^{p_{j_1}}},
$$
and the result follows for $N=2$.

\par

Next suppose that $N\ge 3$, and that the result holds for less numbers
of factors in the convolution. Since the convolution is commutative, we
may assume that $p_N<1$ is the smallest number in $I_N$. Then
$p_N\le p_0$, and \eqref{Eq:LebExpYoung} is the same as
$$
\frac 1{p_0} \le \sum _{j=1}^{N-1} \frac 1{p_j} - Q_{N-1}(p_1,\dots ,p_{N-1}).
$$
By the induction hypothesis we get
$$
\nm {a_1*\cdots *a_{N-1}}{\ell ^{p_0}}
\le
\nm {a_1}{\ell ^{p_1}}\cdots \nm {a_{N-1}}{\ell ^{p_{N-1}}}.
$$
Since $p_N<1$ and $p_N\le p_0$ we get
$$
\nm {a_1*\cdots *a_N}{\ell ^{p_0}}
\le
\nm {a_1*\cdots *a_{N-1}}{\ell ^{p_0}}\nm {a_N}{\ell ^{p_N}}
\le
\nm {a_1}{\ell ^{p_1}}\cdots \nm {a_N}{\ell ^{p_N}}. \qedhere
$$ 
\end{proof}

\par

\end{document}